\documentclass{amsart}
\usepackage{tikz,tikz-cd}
\usepackage{url}
\usepackage{graphicx}
\usepackage{color}
\usepackage{times}
\usepackage{enumerate,latexsym}
\usepackage{amsmath,amssymb}
\usepackage{amsthm}
\usepackage{verbatim}
\usepackage{mathrsfs}
\newtheorem{thm}{Theorem}[section]
\newtheorem*{theorem*}{Theorem}
\newtheorem*{acknowledgement*}{Acknowledgement}
\newtheorem{cor}[thm]{Corollary}

\newtheorem{lem}[thm]{Lemma}
\newtheorem{prop}[thm]{Proposition}
\theoremstyle{definition}
\newtheorem{defn}[thm]{Definition}
\theoremstyle{remark}
\newtheorem{rem}[thm]{Remark}

\numberwithin{equation}{section}

\newcommand{\set}[1]{\left\{#1\right\}}
\newcommand{\Real}{\mathbb R}

\newcommand{\func}[1]{\ensuremath{\mathop{\mathrm{#1}}} }

\newcommand{\spt}[0]{\func{spt}}

\newcommand{\cC}[0]{\mathcal{C}}

\newcommand{\OO}{\mathbf{0}}

\newcommand{\rstar}[1]{(\star_{#1})}

\title{Topological uniqueness for self-expanders of small entropy}

\author{Jacob Bernstein}
\address{Department of Mathematics, Johns Hopkins University, 3400 N. Charles Street, Baltimore, MD 21218}
\email{bernstein@math.jhu.edu}

\author{Lu Wang}
\address{Department of Mathematics, California Institute of Technology, 1200 E. California Boulevard, Pasadena, CA 91125}
\email{drluwang@caltech.edu}

\begin{document}

\begin{abstract}
For a fixed regular cone in Euclidean space with small entropy we show that all smooth self-expanding solutions of the mean curvature flow that are asymptotic to the cone are in the same isotopy class.
\end{abstract}

\maketitle

\section{Introduction} \label{IntroSec}
A \emph{hypersurface}, i.e., a properly embedded codimension-one submanifold, $\Sigma\subset\mathbb{R}^{n+1}$, is a \emph{self-expander} if
\begin{equation} \label{ExpanderEqn}
\mathbf{H}_\Sigma=\frac{\mathbf{x}^\perp}{2}.
\end{equation}
Here 
$$
\mathbf{H}_\Sigma=\Delta_\Sigma\mathbf{x}=-H_\Sigma\mathbf{n}_\Sigma=-\mathrm{div}_\Sigma(\mathbf{n}_\Sigma)\mathbf{n}_\Sigma
$$
is the mean curvature vector, $\mathbf{n}_\Sigma$ is the unit normal, and $\mathbf{x}^\perp$ is the normal component of the position vector. Self-expanders arise naturally in the study of mean curvature flow. Indeed, $\Sigma$ is a self-expander if and only if the family of homothetic hypersurfaces
$$
\left\{\Sigma_t\right\}_{t>0}=\left\{\sqrt{t}\, \Sigma\right\}_{t>0}
$$
is a \emph{mean curvature flow} (MCF), that is, a solution to the flow
$$
\left(\frac{\partial \mathbf{x}}{\partial t}\right)^\perp=\mathbf{H}_{\Sigma_t}.
$$
Self-expanders  model the behavior of a MCF as it emerges from a conical singularity \cite{AIC}. They also model possible long time behavior of the flow \cite{EHAnn}.

Given a hypersurface $\Sigma\subset\mathbb{R}^{n+1}$ the \emph{Gaussian surface area of $\Sigma$} is 
$$
F[\Sigma]=(4\pi)^{-\frac{n}{2}} \int_\Sigma e^{-\frac{|\mathbf{x}|^2}{4}} \, d\mathcal{H}^n
$$
where $\mathcal{H}^n$ denotes the $n$-dimensional Hausdorff measure. In \cite{CMGen}, Colding and Minicozzi introduced a notion of entropy for hypersurfaces which is given by
$$
\lambda[\Sigma]=\sup_{\mathbf{y}\in\mathbb{R}^{n+1},\rho>0} F[\rho\Sigma+\mathbf{y}].
$$
Entropy is invariant under dilations and translations and is a natural measure of geometric complexity; see \cite{BWInvent}, \cite{BWDuke}, \cite{BWGT}, \cite{BWMRL}, \cite{BSW}, \cite{CIMW}, \cite{HershkovitsWhite}, \cite{KetoverZhou}, \cite{MramorWang}, \cite{SWang} and \cite{JZhu}. It follows from Huisken's monotonicity formula  \cite{Huisken} that entropy is non-increasing under the MCF. It is easily checked that $\lambda[\mathbb{R}^n]=1$.  Moreover,  by computations of Stone \cite{Stone},
$$
2>\lambda[\mathbb{S}^1]>\frac{3}{2}>\lambda[\mathbb{S}^2]>\cdots>\lambda[\mathbb{S}^{n}]>\lambda[\mathbb{S}^{n+1}]>\cdots\to\sqrt{2}.
$$

Given an integer $k\geq 2$, $\Sigma$ is a \emph{$C^{k}$-asymptotically conical hypersurface in $\mathbb{R}^{n+1}$} with asymptotic cone $\cC=\mathcal{C}(\Sigma)$ if $\lim_{\rho\to 0^+} \rho \Sigma=\cC$ in $C^{k}_{loc} (\mathbb{R}^{n+1}\setminus\{\mathbf{0}\})$, where $\cC$ is a $C^k$-regular cone. Let $\mathcal{L}(\Sigma)=\mathcal{L}(\cC)=\cC\cap \mathbb{S}^n$ be the \emph{link} of the asymptotic cone, and observe that $\mathcal{L}(\Sigma)$ is a $C^k$-hypersurface in $\mathbb{S}^n$. If $\Sigma$ is a $C^2$-asymptotically conical self-expander, then it follows from Huisken's monotonicity formula and the lower semi-continuity of entropy that $\lambda[\Sigma]=\lambda[\mathcal{C}(\Sigma)]$ -- see, for instance, \cite[Lemma 3.5]{BWProperness}.

Self-expanders in $\mathbb{R}^2$ have been studied in work of Ecker-Huisken \cite{EHAnn} and so we restrict attention to $n\geq 2$.  It can be readily shown, e.g., \cite{EHAnn}, that for a smooth graphical cone, $\mathcal{C}$, there is a unique self-expander asymptotic to $\mathcal{C}$. In contrast, in \cite[Section 8]{BWDegree} (cf. \cite{AIC}), we showed that there is an open subset in the space of regular cones in $\mathbb{R}^3$ so that for any cone in the subset there are at least three distinct self-expanders asymptotic to the cone -- two that are topologically annuli and one that is a pair of disks. Our main result is that this topological non-uniqueness cannot occur for self-expanders that are asymptotic to a low entropy cone.

\begin{thm} \label{MainThm}  
For $k\geq 2$ and $2\leq n \leq 6$, let $\cC$ be a $C^{k+1}$-regular cone in $\Real^{n+1}$ that satisfies
$$
\lambda[\cC]<\lambda[\mathbb{S}^{n-1}\times \Real].
$$
If $\Gamma_1, \Gamma_2$ are both $C^{k+1}$-asymptotically conical self-expanders with $\mathcal{C}(\Gamma_1)=\mathcal{C}(\Gamma_2)=\cC$, then $\Gamma_1$ and $\Gamma_2$ are $C^k$ a.c.-isotopic with fixed cone.
\end{thm}

Here two asymptotically conical hypersurfaces are said to be \emph{a.c.-isotopic with fixed cone} if there is an isotopy that respects the asymptotically conical behavior and fixes the asymptotic cone -- see Section \ref{IsotopySubsec} for the precise definition. In particular, $\Gamma_1$ and $\Gamma_2$ are diffeomorphic.

The dimension restriction comes from our use of the regularity theory of stable minimal hypersurfaces. In fact under, additional, possibly stronger assumptions on the entropy of the asymptotic cone, one has the same result in dimension $n\geq 7$.  In order to state this extra assumption, first let $\mathcal{RMC}_n$ denote the space of \emph{regular minimal cones} in $\Real^{n+1}$, that is $\cC\in \mathcal{RMC}_n$ if and only if it is a proper subset of $\Real^{n+1}$ and $\cC$ is a hypersurface in $\Real^{n+1}\setminus\set{\OO}$ that is invariant under dilation about $\OO$ and with vanishing mean curvature. Let $\mathcal{RMC}_n^*$ denote the set of non-flat elements of $\mathcal{RMC}_n$ -- i.e., cones with non-zero curvature somewhere. For any $\Lambda>0$, let 
$$
\mathcal{RMC}_n(\Lambda)=\set{\cC\in \mathcal{RMC}_n\colon \lambda[\cC]< \Lambda} \mbox{ and } \mathcal{RMC}_n^*(\Lambda)=\mathcal{RMC}^*_n \cap \mathcal{RMC}_n(\Lambda).
$$
Now fix a dimension $n\geq 3$ and a value $ \Lambda>1$.  Consider the following hypothesis:
\begin{equation} \label{Assump1}
\mbox{For all $3\leq l\leq n$, }\mathcal{RMC}_{l}^*(\Lambda)=\emptyset \tag{$\star_{n,\Lambda}$}.
\end{equation}
Observe that all regular minimal cones in $\mathbb{R}^2$ consist of unions of rays and so $\mathcal{RMC}^*_1=\emptyset$. Likewise, as great circles are the only closed geodesics in $\mathbb{S}^2$, $\mathcal{RMC}_2^*=\emptyset$. As a consequence of Allard's regularity theorem and a dimension reduction argument, there is always some $\Lambda>1$ so that $\rstar{n,\Lambda}$ holds. Let
$$
\Lambda_n=\sup \set{ \Lambda\in (1, 2)\colon \rstar{n,\Lambda}\mbox{ holds}}
$$
and
$$
\Lambda_n^*=\left\{ \begin{array}{cc} \lambda[\mathbb{S}^{n-1}\times \Real] & 2\leq n \leq 6 \\
\min\set{\Lambda_n, \lambda[\mathbb{S}^{n-1}\times\Real]} & n\geq 7 \end{array}.
\right.
$$
Observe that $2=\Lambda_2>\Lambda_1^*=\lambda[\mathbb{S}^1\times \Real]$ and that it follows from Marques-Neves's \cite[Theorem B]{MarquesNeves} proof of the Willmore conjecture that $2>\Lambda_3>\lambda[\mathbb{S}^2\times \Real]$ and so it is possible that $\Lambda_n^*= \lambda[\mathbb{S}^{n-1}\times \Real]$ for all $n$. However, this is still an open question when $n\geq 4$.

Using $\Lambda_n^*$, we are able to generalize Theorem \ref{MainThm} to all dimensions.

\begin{thm} \label{AuxThm}  
For any $k,n\geq 2$, let $\cC\subset \Real^{n+1}$ be a $C^{k+1}$-regular cone that satisfies
$$
\lambda[\cC]<\Lambda_n^*.
$$
If $\Gamma_1, \Gamma_2$ are both $C^{k+1}$-asymptotically conical self-expanders with $\mathcal{C}(\Gamma_1)=\mathcal{C}(\Gamma_2)=\cC$, then $\Gamma_1$ and $\Gamma_2$ are $C^k$ a.c.-isotopic with fixed cone.
\end{thm}  

Next we discuss applications of Theorem \ref{AuxThm}. First we observe that Theorem \ref{AuxThm} implies that low entropy cones with disconnected link can't resolve into connected self-expanders.  

\begin{cor}\label{ConnectCor}
For $k,n\geq 2$, let $\cC\subset \Real^{n+1}$ be a $C^{k+1}$-regular cone with $\lambda[\cC]<\Lambda_n^*$. If $\mathcal{L}(\cC)$ has $m$ connected components, then any asymptotically conical self-expander $\Gamma$ with $\cC(\Gamma)=\cC$ has exactly $m$ connected components.
\end{cor}

\begin{rem}
Examples of Angenent-Ilmanen-Chopp \cite{AIC} and Bernstein-Wang \cite{BWDegree} show that there are many cones $\cC$ with disconnected link which flow into connected self-expanders. Numerical computations also show that there are rotationally symmetric double cones in $\Real^3$ that have entropy below $\Lambda_2^*=\lambda[\mathbb{S}^1\times \Real]$.
\end{rem}

\begin{proof}[Proof of Corollary \ref{ConnectCor}]
Let $\sigma_1, \ldots, \sigma_m$ be the connected components of $\mathcal{L}(\cC)$ and let $\cC_i=\cC[\sigma_i]$ be the corresponding cones. Observe that $\lambda[\cC_i]\leq \lambda[\cC]<\Lambda_n^*$. By a minimization procedure sketched by Ilmanen \cite{IlmanenLec} (see Ding \cite[Theorem 6.3]{Ding} for full details), a dimension reduction argument \cite[Theorem 4]{WhiteStrata} and Allard's regularity theorem \cite[Theorem 24.4]{Simon}, there is a self-expander $\Gamma_i^\prime$ asymptotic to $\cC_i$. As each $\sigma_i$ is connected and there are no closed self-expanders, each $\Gamma_i^\prime$ is connected.  Set $\Gamma^\prime=\bigcup_{i=1}^m \Gamma_i^\prime$.  Notice that $\Gamma'$ is a (possibly immersed) asymptotically conical self-expander that is asymptotic to $\cC$.  However, $\lambda[\Gamma^\prime]=\lambda[\cC]<\Lambda_n^*<2$ and so $\Gamma^\prime$ is an embedded hypersurface and, hence, has $m$ components.  Hence, Theorem \ref{AuxThm} implies $\Gamma^\prime$ and $\Gamma$ have the same number of components, proving the claim.
\end{proof}

In order to state the second application, we introduce the following notation for the links of $C^k$-regular cones with entropy bounded by $\Lambda$,
$$
\mathcal{S}^k(\Lambda)=\set{\sigma\subset \mathbb{S}^n\colon \sigma \mbox{ is a $C^k$-hypersurface with } \lambda[\mathcal{C}[\sigma]]<\Lambda}.
$$
Here $\mathcal{C}[\sigma]$ is the cone whose link is equal to $\sigma$. For $\Lambda>1$, let $\mathcal{S}^k_0(\Lambda)\subseteq \mathcal{S}^k(\Lambda)$ denote the set of all such links that are isotopic (inside $\mathcal{S}^k(\Lambda)$) to the equatorial sphere in $\mathbb{S}^n$. We prove existence and topological uniqueness results for asymptotically conical self-expanders with asymptotic link in $\mathcal{S}^{k+1}_0(\Lambda_n^*)$ for $k,n\geq 2$. When $3\leq n\leq 6$ this entails a new existence result for topologically trivial self-expanders asymptotic to small entropy cones. 

\begin{cor}\label{Uniq2Cor}
For any $k,n\geq 2$, if $\sigma\in \mathcal{S}^{k+1}_0(\Lambda_n^*)$, then there is a $C^k$-asymptotically conical self-expander $\Gamma$ with $\mathcal{L}(\Gamma)=\sigma$. Moreover, any such $\Gamma$ is $C^k$ a.c.-isotopic to $\Real^n\times \set{0}$. 	
\end{cor}

\begin{proof}
First of all, by the maximum principle, the only self-expanders that are asymptotic to a given hyperplane is the hyperplane itself. When $n=2$ or $n\geq 7$ the existence of at least one self-expander of the desired topological type is then an immediate consequence of \cite[Theorems 1.1 and 1.2]{BWProperness} and the existence of a $\mathbb{Z}_2$-degree \cite[Corollary 1.3]{BWBanach}. When $3\leq n \leq 6$ one uses Theorem \ref{StableIsotopyThm} to see that there is always at least one $C^k$-asymptotically conical stable self-expander $\Gamma$ with $\mathcal{C}(\Gamma)=\cC$ that is $C^k$ a.c.-isotopic to the hyperplane for any $\cC$ with $\mathcal{L}(\cC)\in \mathcal{S}^{k+1}_0(\Lambda_n^*)$ -- see Appendix \ref{IsotopyTrivialApp} for details. The topological uniqueness follows directly from Theorem \ref{AuxThm}.
\end{proof}

A final application is to the topological properties of closed hypersurfaces of small entropy. It is known by work of ourselves \cite{BWInvent} and J. Zhu \cite{JZhu} that round spheres uniquely minimize the entropy within the class of closed hypersurfaces in $\mathbb{R}^{n+1}$.  In \cite{BWDuke}, we classify all low entropy self-shrinkers in $\Real^3$ and, as a consequence, show that any closed surface in $\Real^3$ of sufficiently small entropy is isotopic, via a MCF, to the round sphere.  This argument is specific to $n=2$ as such a complete classification of self-shrinkers is not known in higher dimensions. However, using a weak flow and a topological classification of low entropy self-shrinkers in $\Real^4$, we show, in \cite{BWGT}, that any closed hypersurface in $\Real^4$ of sufficiently small entropy is diffeomorphic to $\mathbb{S}^3$.  In  \cite{BWIsotopy}, we combine Theorem \ref{AuxThm} with the weak flow of \cite{BWGT} to prove a stronger topological stability theorem. Namely, that any closed hypersurface in $\mathbb{R}^4$ with entropy less than or equal to that of the round cylinder is isotopic to the standard $\mathbb{S}^3$. That is, the $4$-dimensional smooth Schoenflies conjecture holds for closed hypersurfaces in $\mathbb{R}^4$ of low entropy.

The paper is organized as follows: In Section \ref{NotationSec} we fix the notation for the remainder of the paper and discuss background about the question under consideration. In Section \ref{UniversalBarrierSec} we construct a universal barrier which is used in later sections to show the existence of self-expanders with prescribed asymptotic cones. In Section \ref{PartialOrderSec} we introduce a natural partial order on the space of asymptotically conical self-expanders and prove the existence and uniqueness of the greatest and least elements. In Section \ref{FlowSec} we investigate properties of the MCF starting from an asymptotically conical hypersurface of low entropy that is expander mean-convex, and show that such a hypersurface is a.c.-isotopic with fixed cone, via the flow, to a stable self-expander. In Section \ref{DeformUnstableExpanderSec} we use a perturbation by the first eigenfunction of the stability operator for self-expanders together with results of the preceding section to deform any low entropy asymptotically conical unstable self-expander, in the a.c.-isotopy class and preserving the asymptotic cone, to a stable self-expander. In Section \ref{PerturbStableExpanderSec} we apply the analysis carried out in our previous work \cite{BWBanach} and results from Section \ref{FlowSec} to show that one may connect, via an a.c.-isotopy that does not move the asymptotic cones much along the path, any weakly stable self-expander to a self-expander asymptotic to a cone which is a generic perturbation of the asymptotic cone of the initial self-expander. In Section \ref{ProofMainThmSec} we complete the proof of Theorem \ref{MainThm} and Theorem \ref{AuxThm}.

\subsection*{Acknowledgements}
The first author was partially supported by the NSF Grants DMS-1609340 and DMS-1904674 and the Institute for Advanced Study with funding provided by the Charles Simonyi Endowment. The second author was partially supported by an Alfred P. Sloan Research Fellowship, the NSF Grants DMS-2018220 (formerly DMS-1811144) and DMS-2018221 (formerly DMS-1834824), the Office of the Vice Chancellor for Research and Graduate Education at University of Wisconsin-Madison with funding from the Wisconsin Alumni Research Foundation, a Vilas Early Career Investigator Award, and a von Neumann Fellowship by the Institute for Advanced Study with funding from the Z\"{u}rich Insurance Company and the NSF Grant DMS-1638352.

\section{Background and notation} \label{NotationSec}
For the reader's convenience, we recall, in  Sections \ref{NotionSubsec}-\ref{IndexSubsec}, some of the notation and background introduced in our previous works \cite{BWBanach,BWDegree}. In Section \ref{IsotopySubsec} we define an a.c.-isotopy between two asymptotically conical hypersurfaces and discuss some basic properties of a.c.-isotopies.

\subsection{Basic notions} \label{NotionSubsec}
Denote a (open) ball in $\mathbb{R}^n$ of radius $R$ and center $x$ by $B_R^n(x)$ and the closed ball by $\bar{B}^n_R(x)$. We often omit the superscript, $n$, when its value is clear from context. We also omit the center when it is the origin. Given a set $K\subseteq \Real^{n+1}$ the closure of $K$ is denoted by $\mathrm{cl}(K)$ and the \emph{$r$-tubular neighborhood} of $K$ is
$$
\mathcal{T}_r(K)=\bigcup_{p\in K} B_r(p).
$$

For an open subset $U\subseteq\mathbb{R}^{n+1}$, a \emph{(smooth) hypersurface in $U$}, $\Sigma$, is a properly embedded, codimension-one smooth submanifold of $U$. We also consider hypersurfaces of lower regularity and given an integer $k\geq 2$ and $\alpha\in [0,1)$ we define a \emph{$C^{k,\alpha}$-hypersurface in $U$} to be a properly embedded, codimension-one $C^{k,\alpha}$ submanifold of $U$. When needed, we distinguish between a point $p\in\Sigma$ and its \emph{position vector} $\mathbf{x}(p)$.

Consider the hypersurface $\mathbb{S}^n\subset\mathbb{R}^{n+1}$, the unit $n$-sphere in $\mathbb{R}^{n+1}$. For $n\geq 2$, a \emph{(smooth) hypersurface in $\mathbb{S}^n$}, $\sigma$, is a closed, embedded, codimension-one smooth submanifold of $\mathbb{S}^n$ and \emph{$C^{k,\alpha}$-hypersurfaces in $\mathbb{S}^n$} are defined likewise. Observe that $\sigma$ is a closed codimension-two submanifold of $\mathbb{R}^{n+1}$ and so we may associate to each point $p\in\sigma$ its position vector $\mathbf{x}(p)$. Clearly, $|\mathbf{x}(p)|=1$.

A \emph{cone} is a set $\cC\subset\mathbb{R}^{n+1}\setminus\{\mathbf{0}\}$ that is dilation invariant around the origin. That is, $\rho\mathcal{C}=\mathcal{C}$ for all $\rho>0$. The \emph{link} of the cone is the set $\mathcal{L}(\mathcal{C})=\cC\cap\mathbb{S}^{n}$. The cone is \emph{regular} if its link is a smooth hypersurface in $\mathbb{S}^{n}$ and \emph{$C^{k,\alpha}$-regular} if its link is a $C^{k,\alpha}$-hypersurface in $\mathbb{S}^n$. For any hypersurface $\sigma\subset\mathbb{S}^n$ the \emph{cone over $\sigma$}, $\mathcal{C}[\sigma]$, is the cone defined by 
$$
\mathcal{C}[\sigma]=\left\{\rho p\colon p\in\sigma, \rho>0\right\}\subset\mathbb{R}^{n+1}\setminus\{\mathbf{0}\}.
$$
Clearly, $\mathcal{L}(\mathcal{C}[\sigma])=\sigma$. 

\subsection{Function spaces} \label{FunctionSubsec}
Let $\Sigma$ be a properly embedded, $C^{k,\alpha}$ submanifold of an open subset $U\subseteq\mathbb{R}^{n+1}$. There is a natural Riemannian metric, $g_\Sigma$, on $\Sigma$ of class $C^{k-1,\alpha}$ induced from the Euclidean one. As we always take $k\geq 2$, the Christoffel symbols of this metric, in appropriate coordinates, are well-defined and of regularity $C^{k-2,\alpha}$. Let $\nabla_\Sigma$ be the covariant derivative on $\Sigma$. Denote by $d_\Sigma$ the geodesic distance on $\Sigma$ and by $B^\Sigma_\rho(p)$ the (open) geodesic ball in $\Sigma$ of radius $\rho$ and center $p\in\Sigma$. For $\rho$ small enough so that $B_{\rho}^\Sigma(p)$ is strictly geodesically convex and $q\in B^\Sigma_\rho(p)$, denote by $\tau^\Sigma_{p,q}$ the parallel transport along the unique minimizing geodesic in $B^\Sigma_\rho(p)$ from $p$ to $q$. 

Throughout the rest of this section, let $\Omega$ be a domain in $\Sigma$, and let $l$ be an integer in $[0,k]$, $\gamma\in (0,1)$ and $d\in\mathbb{R}$. Suppose $l+\gamma\leq k+\alpha$. We first consider the following norm for functions on $\Omega$:
$$
\Vert f\Vert_{l; \Omega}=\sum_{i=0}^l \sup_{\Omega} |\nabla_\Sigma^i f|.
$$
We then let
$$
C^l(\Omega)=\left\{f\in C_{loc}^l(\Omega)\colon \Vert f\Vert_{l; \Omega}<\infty\right\}.
$$
We next define the H\"{o}lder semi-norms for functions $f$ and tensor fields $T$ on $\Omega$: 
$$
[f]_{\gamma; \Omega} =\sup_{\substack{p,q\in\Omega \\ q\in B^\Sigma_{\delta}(p)\setminus\{p\}}} \frac{|f(p)-f(q)|}{d_\Sigma(p,q)^\gamma} 
\mbox{ and } 
[T]_{\gamma; \Omega} =\sup_{\substack{p,q\in\Omega \\ q\in B^\Sigma_{\delta}(p)\setminus\{p\}}} \frac{|T(p)-(\tau^\Sigma_{p,q})^* T(q)|}{d_\Sigma(p,q)^\gamma},
$$
where $\delta=\delta(\Sigma,\Omega)>0$ so that, for all $p\in\Omega$, $B^\Sigma_\delta(p)$ is strictly geodesically convex. We further define the norm for functions on $\Omega$:
$$
\Vert f\Vert_{l, \gamma; \Omega}=\Vert f\Vert_{l; \Omega}+[\nabla_\Sigma^l f]_{\gamma; \Omega},
$$
and let 
$$
C^{l, \gamma}(\Omega)=\left\{f\in C_{loc}^{l, \gamma}(\Omega)\colon \Vert f\Vert_{l, \gamma; \Omega}<\infty\right\}.
$$

We also define the following weighted norm for functions on $\Omega$:
$$
\Vert f\Vert_{l; \Omega}^{(d)}=\sum_{i=0}^l\sup_{p\in\Omega} \left(|\mathbf{x}(p)|+1\right)^{-d+i} |\nabla_\Sigma^i f(p)|.
$$
We then let 
$$
C^{l}_d(\Omega)=\left\{f\in C^l_{loc}(\Omega)\colon \Vert f\Vert_{l; \Omega}^{(d)}<\infty\right\}.
$$
We further define the following weighted H\"{o}lder semi-norms for functions $f$ and tensor fields $T$ on $\Omega$:
\begin{align*}
[f]_{\gamma; \Omega}^{(d)} & =\sup_{\substack{p,q\in\Omega \\ q\in B^\Sigma_{\delta_p}(p)\setminus\{p\}}} \left((|\mathbf{x}(p)|+1)^{-d+\gamma}+(|\mathbf{x}(q)|+1)^{-d+\gamma}\right) \frac{|f(p)-f(q)|}{d_\Sigma(p,q)^\gamma}, \mbox{ and}, \\
[T]_{\gamma; \Omega}^{(d)} & =\sup_{\substack{p,q\in\Omega \\ q\in B^\Sigma_{\delta_p}(p)\setminus\{p\}}} \left((|\mathbf{x}(p)|+1)^{-d+\gamma}+(|\mathbf{x}(q)|+1)^{-d+\gamma}\right) \frac{|T(p)-(\tau^\Sigma_{p,q})^* T(q)|}{d_\Sigma(p,q)^\gamma},
\end{align*}
where $\eta=\eta(\Omega,\Sigma)\in \left(0,\frac{1}{4}\right)$ so that for any $p\in\Sigma$, letting $\delta_p=\eta (|\mathbf{x}(p)|+1)$, $B_{\delta_p}^\Sigma(p)$ is strictly geodesically convex. Next we define the norm for functions on $\Omega$:
$$
\Vert f\Vert_{l, \gamma; \Omega}^{(d)}=\Vert f\Vert_{l; \Omega}^{(d)}+[\nabla_\Sigma^l f]_{\gamma; \Omega}^{(d-l)},
$$
and we let
$$
C^{l,\gamma}_d(\Omega)=\left\{f\in C^{l,\gamma}_{loc}(\Omega)\colon \Vert f\Vert_{l, \gamma; \Omega}^{(d)}<\infty\right\}.
$$
We follow the convention that $C^{l,0}_{loc}=C^{l}_{loc}$, $C^{l,0}=C^l$ and $C^{l,0}_d=C^l_d$ and that $C^{0, \gamma}_{loc}=C^\gamma_{loc}$,  $C^{0,\gamma}=C^\gamma$ and $C^{0,\gamma}_d=C^\gamma_d$. The notation for the corresponding norms is abbreviated in the same fashion.

In all above definitions of various norms, we often omit the domain $\Omega$ when it is clear from context. These norms can be extended in a straightforward manner to vector-valued functions and tensor fields.  It is a standard exercise to verify that these spaces equipped with the corresponding norms are Banach spaces.

\subsection{Asymptotically conical hypersurfaces} \label{ACHSubsec}
For $k,n\geq 2$ and $\alpha\in [0,1)$, let $\mathcal{C}$ be a $C^{k,\alpha}$-regular cone in $\mathbb{R}^{n+1}$. Let $\mathbf{V}\colon\mathcal{C}\to\mathbb{R}^{n+1}$ be a homogeneous transverse section on $\cC$, that is a $C^{k,\alpha}$ vector field along $\mathcal{C}$ so that 
\begin{itemize}
\item $|\mathbf{V}|=1$;
\item $\mathbf{V}(p)$ does not lie in $T_p\mathcal{C}$;
\item $\mathbf{V}(\rho p)=\mathbf{V}(p)$ for all $\rho>0$ and $p\in\mathcal{C}$.
\end{itemize}
A $C^{k,\alpha}$-hypersurface, $\Sigma\subset\mathbb{R}^{n+1}$, is \emph{$C^{k,\alpha}_{*}$-asymptotically conical} with asymptotic cone $\mathcal{C}$ if there is a radius $R>1$ and a function $u\in C^{k,\alpha}_1(\mathcal{C}\setminus\bar{B}_R)$ with
$$
\lim_{\rho\to 0^+} \rho u(\rho^{-1} p)=0 \mbox{ in $C^k_{loc}(\mathcal{C})$}
$$
so that 
$$
\Sigma\setminus\bar{B}_{2R}\subset\set{\mathbf{x}(p)+u(p)\mathbf{V}(p)\colon p\in\mathcal{C}\setminus\bar{B}_R}\subset\Sigma
$$
When $\alpha=0$ this means that $\Sigma$ is $C^k$-asymptotically conical as defined in Section \ref{IntroSec}. As observed in \cite{BWBanach}, this definition is independent of the choice of $\mathbf{V}$. Clearly, the asymptotic cone, $\cC$, is uniquely determined by $\Sigma$ and so we denote it by $\mathcal{C}(\Sigma)$ and its link by $\mathcal{L}(\Sigma)$. Denote the space of $C^{k,\alpha}_{*}$-asymptotically conical $C^{k,\alpha}$-hypersurfaces in $\mathbb{R}^{n+1}$ by $\mathcal{ACH}^{k,\alpha}_n$. As before we denote $\mathcal{ACH}^{k}_n=\mathcal{ACH}^{k,0}_n$.

If $\Sigma\in\mathcal{ACH}^2_n$ is a self-expander, then the interior estimates for MCF (see \cite{EHInvent}) imply that for all $i\geq 0$
\begin{equation} \label{CurvExpanderEqn}
\sup_{p\in\Sigma} (1+|\mathbf{x}(p)|) |\nabla^i A_\Sigma|<\infty.
\end{equation}

\subsection{Traces at infinity} \label{TraceSubsec}
Fix an element $\Sigma\in\mathcal{ACH}_n^{k,\alpha}$ and let $\mathbf{V}$ be a homogeneous transverse section on the asymptotic cone $\mathcal{C}(\Sigma)$. If $\pi_{\mathbf{V}}$ denotes the projection to $\mathcal{C}(\Sigma)$ along $\mathbf{V}$, then $\pi_{\mathbf{V}}$ restricts to a $C^{k,\alpha}$ diffeomorphism of $\Sigma^\prime=\Sigma\setminus K$ for some compact set $K$ onto $\mathcal{C}(\Sigma)\setminus\bar{B}_R$ and denote its inverse by $\theta_{\mathbf{V};\Sigma^\prime}$. Let $l\geq 0$ be an integer and $\gamma\in [0,1)$ such that $l+\gamma\leq k+\alpha$. 

A map $\mathbf{f}\in C_{loc}^{l,\gamma}(\Sigma; \mathbb{R}^M)$ is \emph{asymptotically homogeneous of degree $d$} if 
$$
\lim_{\rho\to 0^+} \rho^d \mathbf{f}\circ\theta_{\mathbf{V};\Sigma^\prime}(\rho^{-1} p)=\mathbf{g}(p) \mbox{ in $C_{loc}^{l,\gamma}(\mathcal{C}(\Sigma);\mathbb{R}^M)$}
$$
where $\mathbf{g}$ is homogeneous of degree $d$, i.e., $\rho^d\mathbf{g}(\rho^{-1} p)=\mathbf{g}(p)$ for all $\rho>0$ and $p\in\mathcal{C}(\Sigma)$. For such $\mathbf{f}$ we define the \emph{trace at infinity of $\mathbf{f}$} by
$$
\mathrm{tr}_\infty^d[\mathbf{f}]=\mathbf{g}|_{\mathcal{L}(\Sigma)}\in C^{l,\gamma}(\mathcal{L}(\Sigma);\mathbb{R}^M).
$$
Whether $\mathbf{f}$ is asymptotically homogeneous of degree $d$ and the definition of $\mathrm{tr}_{\infty}^d$ are independent of the choice of $\mathbf{V}$. Clearly, $\mathbf{x}|_{\Sigma}$ is asymptotically homogeneous of degree one and $\mathrm{tr}_\infty^{1}[\mathbf{x}|_\Sigma]=\mathbf{x}|_{\mathcal{L}(\Sigma)}$.

We next define the space
$$
C_{d,\mathrm{H}}^{l,\gamma}(\Sigma; \mathbb{R}^M)=\left\{\mathbf{f}\in C_{d}^{l,\gamma}(\Sigma; \mathbb{R}^M)\colon \mbox{$\mathbf{f}$ is asymptotically homogeneous of degree $d$}\right\}.
$$
One can check that $C_{d,\mathrm{H}}^{l,\gamma}(\Sigma; \mathbb{R}^M)$ is a closed subspace of $C_d^{l,\gamma}(\Sigma; \mathbb{R}^M)$ and the map 
$$
\mathrm{tr}_{\infty}^d\colon C_{d,\mathrm{H}}^{l,\gamma}(\Sigma; \mathbb{R}^M)\to C^{l,\gamma}(\mathcal{L}(\Sigma); \mathbb{R}^M)
$$
is a bounded linear map. We further define the set $C^{l,\gamma}_{d,0}(\Sigma;\mathbb{R}^M)\subset C_{d,\mathrm{H}}^{l,\gamma}(\Sigma; \mathbb{R}^M)$ to be the kernel of $\mathrm{tr}_\infty^d$.

\subsection{Asymptotically conical embeddings} \label{ACESubsec}
Fix an element $\Sigma\in \mathcal{ACH}^{k,\alpha}_n$. Given a map $\varphi\colon\mathcal{L}(\Sigma)\to\mathbb{R}^{n+1}$, the \emph{homogeneous extension of degree one of $\varphi$} is the map $\mathscr{E}^\mathrm{H}_1[\varphi]\colon\mathcal{C}(\Sigma)\to\mathbb{R}^{n+1}$ defined by 
$$
\mathscr{E}^\mathrm{H}_1[\varphi](p)=|\mathbf{x}(p)| \varphi(|\mathbf{x}(p)|^{-1}p).
$$
We define the space of $C^{k,\alpha}_{*}$-asymptotically conical embeddings of $\Sigma$ into $\mathbb{R}^{n+1}$ to be
$$
\mathcal{ACH}^{k,\alpha}_n(\Sigma)=\left\{\mathbf{f}\in C_{1}^{k,\alpha}\cap C^k_{1,\mathrm{H}}(\Sigma; \mathbb{R}^{n+1})\colon\mbox{$\mathbf{f}$ and $\mathscr{E}^\mathrm{H}_1[\mathrm{tr}_\infty^1[\mathbf{f}]]$ are embeddings}\right\}.
$$
Clearly, $\mathcal{ACH}^{k,\alpha}_n(\Sigma)$ is an open subset of the Banach space $C^{k,\alpha}_{1}\cap C^{k}_{1,\mathrm{H}}(\Sigma; \mathbb{R}^{n+1})$ with the $C^{k,\alpha}_1$ norm. For $\mathbf{f}\in\mathcal{ACH}^{k,\alpha}_n(\Sigma)$, as $\mathscr{E}^\mathrm{H}_1[\mathrm{tr}_\infty^1[\mathbf{f}]]$ is a $C^{k,\alpha}$ embedding that is homogeneous of degree one, it parameterizes the $C^{k,\alpha}$-regular cone $\mathcal{C}(\mathbf{f}(\Sigma))$ -- see  \cite[Proposition 3.3]{BWBanach}.

\subsection{Morse index} \label{IndexSubsec}
We recall the notion of index and nullity for asymptotically conical self-expanders and relate these integers to certain other spectral invariants. First observe that the self-expander equation \eqref{ExpanderEqn} is the Euler-Lagrangian equation for the formally defined functional
$$
E[\Sigma]=\int_\Sigma e^{\frac{|\mathbf{x}|^2}{4}} \, d\mathcal{H}^n.
$$
For a self-expander $\Sigma$, if $\{\Phi_s(\Sigma)\}_{|s|<\epsilon}$ is a compactly supported variation of $\Sigma$ such that ${\frac{d\Phi_s}{ds}\vline}_{s=0}=u\mathbf{n}_\Sigma$, then, by a computation in \cite[Section 4]{BWBanach},
$$
{\frac{d^2}{ds^2}\vline}_{s=0} E[\Phi_s(\Sigma)]=\int_\Sigma \left(|\nabla_\Sigma u|^2+\left(\frac{1}{2}-|A_\Sigma|^2\right)u^2\right) e^{\frac{|\mathbf{x}|^2}{4}} \, d\mathcal{H}^n.
$$
Denote by $Q_\Sigma[u]$ the integral on the right side of the above equation. We define the \emph{(Morse) index} of a self-expander, $\Sigma$, to be
$$
\mathrm{ind}(\Sigma)=\sup \set{\dim V\colon V\subset C^{2}_c(\Sigma) \mbox{ so that } Q_{\Sigma}[u]<0, \forall u\in V\backslash \set{0}}.
$$

In \cite{BWDegree} we introduced a weighted inner product for functions on $\Sigma$,
$$
B_\Sigma[u,v]=\int_\Sigma uv e^{\frac{|\mathbf{x}|^2}{4}} \, d\mathcal{H}^n.
$$
We further showed, in Section 4 of \cite{BWDegree}, that if $\Sigma$ is an asymptotically conical self-expander, then there is a self-adjoint (with respect to $B_\Sigma$) operator 
$$
L_\Sigma=\Delta_\Sigma+\frac{1}{2}\mathbf{x}\cdot\nabla_\Sigma+|A_\Sigma|^2-\frac{1}{2}
$$
so that $Q_\Sigma[u]=-B_\Sigma[u,L_\Sigma u]$ for any functions $u\in C^2_c(\Sigma)$. The operator $L_\Sigma$ has a discrete spectrum with a finite spectral bottom. Thus, $\mathrm{ind}(\Sigma)$ equals the number of negative eigenvalues (counted with multiplicities) of $-L_\Sigma$, and in particular, it is finite.

We also define the \emph{nullity} of an asymptotically conical self-expander $\Sigma$, $\mathrm{null}(\Sigma)$ to be the dimension of the kernel of $L_\Sigma$,
$$
\mathcal{K}_\Sigma=\set{\kappa\in C^2_{loc}\cap C^0_{1,0}(\Sigma)\colon L_\Sigma\kappa=0}.
$$

We call a self-expander \emph{stable} if it has index $0$, and \emph{unstable} otherwise. Moreover, if a stable self-expander has nullity $0$, then we call the self-expander \emph{strictly stable}; otherwise, it is called \emph{weakly stable}.

\subsection{Isotopy} \label{IsotopySubsec}
Two elements $\Sigma_1, \Sigma_2\in \mathcal{ACH}_n^{k,\alpha}$ are \emph{$C^{k,\alpha}$ a.c.-isotopic} if there is a continuous map
$$
\mathbf{F}\colon [0,1]\to \mathcal{ACH}^{k,\alpha}_n(\Sigma_1)
$$
which satisfies $\mathbf{F}(0)=\mathbf{x}|_{\Sigma_1}$ and $\mathbf{F}(1)=\mathbf{f}_1$ with $\mathbf{f}_1(\Sigma_1)=\Sigma_2$. We call $\mathbf{F}$ a \emph{$C^{k,\alpha}$ a.c.-isotopy} between $\Sigma_1$ and $\Sigma_2$. 

An a.c.-isotopy $\mathbf{F}$ between $\Sigma_1$ and $\Sigma_2$, \emph{fixes the asymptotic cone} if 
$$
\frac{\mathrm{tr}_\infty^1[\mathbf{F}(t)]}{|\mathrm{tr}_\infty^1[\mathbf{F}(t)]|}=\mathbf{x}|_{\mathcal{L}(\Sigma_1)} \mbox{ for all $t\in [0,1]$}.
$$ 
If there is an isotopy fixing the asymptotic cone between $\Sigma_1$ and $\Sigma_2$, then we say $\Sigma_1$ and $\Sigma_2$ are \emph{a.c.-isotopic with fixed cone}.

We will use the following lemma repeatedly:

\begin{lem}\label{IsotopyLem}
Let $k,n\geq 2$ and $\alpha\in [0,1)$. If $\Sigma\in \mathcal{ACH}^{k,\alpha}_n$, then there is an $\epsilon_0=\epsilon_0(\Sigma)$ so that if $\mathbf{f}\in \mathcal{ACH}^{k,\alpha}_n(\Sigma)$ satisfies $\Vert \mathbf{f}-\mathbf{x}|_{\Sigma}\Vert^{(1)}_{1}<\epsilon_0$, then the map $\mathbf{F}\colon [0,1]\to\mathcal{ACH}^{k,\alpha}_n(\Sigma)$ defined by $\mathbf{F}(t)=(1-t)\mathbf{x}|_\Sigma+t\mathbf{f}$, provides a $C^{k,\alpha}$ a.c.-isotopy between $\Sigma$ and $\mathbf{f}(\Sigma)$.
\end{lem}

\begin{proof}
The result follows from the implicit function theorem.
\end{proof}

We will also need the following perturbation result which says that any a.c.-isotopy that does not move the asymptotic cones ``too much" along the path can be approximated by an a.c.-isotopy with fixed asymptotic cone.

\begin{lem} \label{IsotopyFixConeLem}
For $k,n\geq 2$ and $\alpha\in [0,1)$, let $\Sigma\in\mathcal{ACH}^{k,\alpha}_n$ and $\varphi\in C^{k,\alpha}(\mathcal{L}(\Sigma);\mathbb{R}^{n+1})$ so that $\mathscr{E}^\mathrm{H}_1[\varphi]$ is an embedding. There is a $\delta_0=\delta_0(\Sigma,\varphi)>0$ and $C_0=C_0(\Sigma)>0$ so that if $\mathbf{F}\colon [0,1]\to\mathcal{ACH}^{k,\alpha}(\Sigma)$ is continuous and, for all $t\in [0,1]$,
$$
\Vert\mathrm{tr}^1_\infty[\mathbf{F}(t)]-\varphi\Vert_{k,\alpha}<\delta_0,
$$
then there is a continuous map $\tilde{\mathbf{F}}\colon [0,1]\to\mathcal{ACH}^{k,\alpha}_n(\Sigma)$ so that, for every $t\in [0,1]$, the following holds:
\begin{enumerate}
\item \label{IsotopyTraceItem} $\mathrm{tr}_\infty^1[\tilde{\mathbf{F}}(t)]=\varphi$;
\item \label{IsotopyErrorItem} $\Vert\tilde{\mathbf{F}}(t)-\mathbf{F}(t)\Vert^{(1)}_{k,\alpha}\leq C_0 \Vert\mathrm{tr}_\infty^1[\mathbf{F}(t)]-\varphi\Vert_{k,\alpha}$.
\end{enumerate}
\end{lem}

\begin{proof}
Let $\mathbf{V}$ be a homogeneous transverse section on $\mathcal{C}(\Sigma)$ and let $\pi_\mathbf{V}$ be the projection of an open neighborhood, $U$, of $\mathcal{C}(\Sigma)$ along $\mathbf{V}$. Define $\mathscr{E}_{\mathbf{V},\Sigma}[\varphi]=\mathscr{E}^\mathrm{H}_1[\varphi]\circ\pi_\mathbf{V}\circ\mathbf{x}|_\Sigma$. There is an $R_\Sigma>1$ so $\Sigma\setminus B_{R_\Sigma}\subset U$ and so $\mathscr{E}_{\mathbf{V},\Sigma}$ is well-defined on $\Sigma\setminus B_{R_\Sigma}$. As $\mathbf{F}$ is continuous and $[0,1]$ is compact, there is an $R>R_\Sigma+1$ and $C=C(\Sigma)>0$ so that, for every $t\in [0,1]$,
\begin{equation} \label{UniformAsympEqn}
\Vert\mathbf{F}(t)-\mathscr{E}_{\mathbf{V},\Sigma}[\varphi]\Vert_{k,\alpha; \Sigma\setminus B_R}^{(1)}<C\delta_0.
\end{equation}
Let $\chi\colon\mathbb{R}^{n+1}\to [0,1]$ be a smooth cut-off function so that $\chi\equiv 1$ outside $B_{4R}$, $\chi\equiv 0$ in $B_{2R}$ and $|D\chi|<2R^{-1}$. Define
$$
\tilde{\mathbf{F}}(t)=\mathbf{F}(t)+(\chi\circ\mathbf{x}|_\Sigma)\mathscr{E}_{\mathbf{V},\Sigma}[\varphi-\mathrm{tr}_\infty^1[\mathbf{F}(t)]].
$$
It is straightforward to verify that $\tilde{\mathbf{F}}(t)\in C^{k,\alpha}_1\cap C^k_{1,\mathrm{H}}(\Sigma;\mathbb{R}^{n+1})$ and Items \eqref{IsotopyTraceItem} and \eqref{IsotopyErrorItem} hold with an appropriate choice of $C_0$. It remains only to show $\tilde{\mathbf{F}}(t)\in\mathcal{ACH}^{k,\alpha}_n(\Sigma)$. To see this one observes that $\tilde{\mathbf{F}}(t)=\mathbf{F}(t)$ on $\Sigma\cap \bar{B}_{2R}$ while, on $\Sigma\setminus B_{2R}$,
$$
\tilde{\mathbf{F}}(t) =\mathscr{E}_{\mathbf{V},\Sigma}[\varphi]+(\mathbf{F}(t)-\mathscr{E}_{\mathbf{V},\Sigma}[\varphi])+ (\tilde{\mathbf{F}}(t)-\mathbf{F}(t)).
$$
Hence, by choosing $\delta_0$ sufficiently small and invoking \eqref{UniformAsympEqn} and Item \eqref{IsotopyErrorItem}, it follows from the implicit function theorem that $\tilde{\mathbf{F}}(t)\in\mathcal{ACH}^{k,\alpha}_n(\Sigma)$, finishing the proof.
\end{proof}

\section{Universal barrier} \label{UniversalBarrierSec}
We prove the following existence of a universal barrier for self-expanders adapted to any $C^2$-regular cone.  In what follows it is helpful to consider the map
$$
\Psi_{\cC}\colon \cC\times \Real\to \Real^{n+1}
$$
associated to a $C^2$-regular cone $\cC$ that is given by
$$
\Psi_{\cC}(p,t)= \cos(t) \mathbf{x}(p)+\sin(t) |\mathbf{x}(p)|\mathbf{n}_{\cC}(p)
$$
where $\mathbf{n}_{\cC}$ is a choice of unit normal on $\cC$. Observe $|\Psi_\cC(p,t)|=|\mathbf{x}(p)|$. As $\mathcal{L}(\cC)$ is of class $C^2$ and  compact, it follows that there is an $\epsilon=\epsilon(\cC)>0$ so that if 
$$
V_{\epsilon, R}(\cC)=\set{(p,t)\in (\cC\backslash \bar{B}_R)\times \Real\colon |t|<\epsilon},
$$
then, for any $R\geq 0$, $\Psi_{\cC}|_{V_{\epsilon,R}(\cC)}$ is a $C^1$ diffeomorphism onto its image. When $R=0$ we simply write $V_\epsilon(\cC)$. 

\begin{prop}\label{UniversalBarrierProp}
For $n\geq 2$, let $\cC\subset \Real^{n+1}$ be a $C^{2}$-regular cone. There exists an open domain $\mathcal{B}(\cC)\subset \Real^{n+1}$, constants $N_0=N_0(\cC)>0$ and $R_0=R_0(\cC)>1+N_0$, and a continuous function $\rho_{\cC}\colon (R_0,\infty)\to \Real^+$ with the following properties:
\begin{enumerate}
\item \label{UnivSubsetItem} $\cC\cup B_{1}\subset \Real^{n+1}\setminus \mathcal{B}(\cC)$;
\item \label{UnivTubItem} For all $R\geq R_0$, $\Real^{n+1}\backslash \left(\mathcal{B}(\cC)\cup \bar{B}_R\right) \subset \mathcal{T}_{N_0 R^{-1}} (\cC)$;
\item \label{UnivRegNbhdItem} If  $V(\cC)=\set{(p,t)\in (\cC\backslash \bar{B}_{R_0}) \times\mathbb{R}\colon |t|\leq \rho_{\cC}(|\mathbf{x}(p)|)}$,
then 
$$
\Psi_{\cC}(V(\cC))=\Real^{n+1}\backslash \left(\mathcal{B}(\cC)\cup \bar{B}_{R_0}\right)
$$
and $\Psi_{\cC}|_{V(\cC)}$ is a $C^1$ diffeomorphism onto its image;
\item \label{UnivCpctBarrItem} If $V$ is an integral $n$-varifold in $\Real^{n+1}$ with compact support and $V$ is $E$-stationary in $\mathcal{B}(\cC)$, then $\spt(V)\cap \mathcal{B}(\cC)=\emptyset$;
\item \label{UnivBarrItem} If $\Sigma \subset \Real^{n+1}$ is an asymptotically conical self-expander with asymptotic cone $\cC$, then $\Sigma\cap \mathcal{B}(\cC)=\emptyset$.
\end{enumerate}
\end{prop}

In order to prove this we first introduce simple barriers modeled on one-sheeted hyperboloids -- see \cite{BWDegree} for a related construction or \cite{Ding} where rotationally symmetric solutions to \eqref{ExpanderEqn} are used instead.

To that end, consider the following family of functions depending on parameters $\mathbf{v}\in \mathbb{S}^n$ and $\eta>0$:
$$
f_{\mathbf{v}, \eta}(\mathbf{x})=2n+ |\mathbf{x}|^2-\left(1+\eta^2\right) (\mathbf{x}\cdot \mathbf{v})^2.
$$
Associated to these functions are the following family of connected closed sets
$$
E_{\mathbf{v}, \eta}=\set{\mathbf{x}\in\mathbb{R}^{n+1}\colon f_{\mathbf{v}, \eta}(\mathbf{x})\leq 0\mbox{ and } \mathbf{x}\cdot \mathbf{v}\geq 0}
$$
and their interiors
$$
E_{\mathbf{v}, \eta}^{\circ} =\mathrm{int}(E_{\mathbf{v}, \eta})=\set{\mathbf{x}\in\mathbb{R}^{n+1}\colon f_{\mathbf{v}, \eta}(\mathbf{x})< 0\mbox{ and } \mathbf{x}\cdot \mathbf{v}> 0}.
$$
Consider the connected, rotationally symmetric cone,
$$
\mathcal{C}_{\mathbf{v}, \eta}=\set{\mathbf{x}\in\mathbb{R}^{n+1}\colon |\mathbf{x}|^2= \left(1+\eta^{2}\right)(\mathbf{x}\cdot \mathbf{v})^2 \mbox{ and } \mathbf{x}\cdot \mathbf{v}> 0},
$$
that lies in the half-space $\set{\mathbf{x\cdot \mathbf{v}\geq 0}}$ and has axis parallel to $\mathbf{v}$ and cone aperture $2\tan^{-1}(\eta)$, and observe that $E_{\mathbf{v}, \eta}$ has boundary asymptotic to $\mathcal{C}_{\mathbf{v}, \eta}$. Moreover, letting
$$
U_{\mathbf{v}, \eta}=\set{\mathbf{x}\in\mathbb{R}^{n+1}\colon |\mathbf{x}|^2< \left(1+\eta^{2}\right)(\mathbf{x}\cdot \mathbf{v})^2 \mbox{ and } \mathbf{x}\cdot \mathbf{v}> 0},
$$
be the open cone that is the component of $\Real^{n+1}\backslash \mathrm{cl}(\mathcal{C}_{\mathbf{v}, \eta})$ that contains $\mathbf{v}$, one has $E_{\mathbf{v}, \eta}\subset U_{\mathbf{v}, \eta}$. By construction, $f_{\mathbf{v}, \eta}>0$ on $\set{|\mathbf{x}\cdot \mathbf{v}|^2< 2n \eta^{-2}}$ and so
$$
E_{\mathbf{v}, \eta}\cap \set{\mathbf{x}\cdot \mathbf{v}<\sqrt{2n} \eta^{-1}}=\emptyset.
$$  

First we show the following asymptotic property for $E_{\mathbf{v},\eta}$:

\begin{lem} \label{BarrierDecayLem}
Given $\eta>0$ there is a unique continuous function $\rho_\eta\colon [\sqrt{2n}\eta^{-1},\infty)\to\mathbb{R}^+$ so that, for every $\mathbf{v}\in\mathbb{S}^n$,
$$
E_{\mathbf{v},\eta}=\set{\Psi_{\mathcal{C}_{\mathbf{v},\eta}}(p,t)\colon p\in\mathcal{C}_{\mathbf{v},\eta}\setminus B_{\sqrt{2n}\eta^{-1}}, t \in [\rho_\eta(|\mathbf{x}(p)|) ,\tan^{-1}(\eta)]}
$$
where we choose $\mathbf{n}_{\mathcal{C}_{\mathbf{v},\eta}}$ so to point into $U_{\mathbf{v},\eta}$. Moreover, 
$$
\sin(\rho_\eta(r)) < 4n\eta^{-1}r^{-2}.
$$
\end{lem}

\begin{proof}
Without loss of generality we assume $\mathbf{v}$ is the north pole of $\mathbb{S}^n$. Denote spherical coordinates on $\mathbb{R}^{n+1}$ by the map $\Phi\colon [0,\infty)\times [0,\pi]\times\mathbb{S}^{n-1}\to\mathbb{R}^{n+1}$ given by
$$
\Phi(r,\tau,\omega)= (r\sin(\tau)\omega,r\cos(\tau)).
$$
Define
$$
C_a=(0,\infty)\times (0,a)\times\mathbb{S}^{n-1}.
$$
Let $\epsilon=\tan^{-1}(\eta)$. As $\cC_{\mathbf{v},\eta}$ is rotationally symmetric, it is straightforward to verify that $\Phi|_{C_\epsilon}$ and $\Psi_{\mathcal{C}_{\mathbf{v},\eta}}|_{V_\epsilon(\cC_{\mathbf{v},\eta})}$ are both $C^1$ diffeomorphisms onto $U_{\mathbf{v},\eta}\setminus(\set{\mathbf{0}}\times\mathbb{R})$. Moreover,
$$
\Psi_{\mathcal{C}_{\mathbf{v},\eta}}^{-1}\circ\Phi|_{C_\epsilon}(r,\tau,\omega)=(\Phi(r,\epsilon,\omega),\epsilon-\tau).
$$
We also observe that $E_{\mathbf{v},\eta}\subset\mathbb{R}^{n+1}\setminus B_{\sqrt{2n}\eta^{-1}}$ and 
$$
E_{\mathbf{v},\eta}\cap (\set{\mathbf{0}}\times\mathbb{R})=\set{\mathbf{0}}\times [\sqrt{2n}\eta^{-1},\infty) =\Psi_{\cC_{\mathbf{v},\eta}}\left((\cC_{\mathbf{v},\eta}\setminus B_{\sqrt{2n}\eta^{-1}})\times\set{\epsilon}\right).
$$

Fix any $r\geq \sqrt{2n}\eta^{-1}$ and $\omega\in\mathbb{S}^{n-1}$. One readily evaluates
\begin{align*}
f(\tau)=f_{\mathbf{v},\eta}(\Phi(r,\tau,\omega)) & =2n+r^2-(1+\eta^2)r^2\cos^2(\tau) \\
& =2n-\eta^2r^2+r^2(1+\eta^2)\sin^2(\tau).
\end{align*}
One notices that $f(\tau)$ is strictly decreasing for $\tau\in [0,\epsilon]$ and $f(0)\leq 0$ while $f(\epsilon)>0$. Thus, there is a unique function $\theta_\eta(r)\in [0,\epsilon)$ so that if $\tau\in [0,\theta_\eta(r)]$ then $f(\tau)\leq 0$ while for $\tau\in (\theta_\eta(r),\epsilon]$ one has $f(\tau)>0$. In fact, 
$$
\theta_\eta(r)=\sin^{-1}\left(\frac{\sqrt{\eta^2-2nr^{-2}}}{\sqrt{1+\eta^2}}\right).
$$
Hence,
$$
E_{\mathbf{v},\eta}=\set{\Phi(r,\tau,\omega)\colon r\geq \sqrt{2n}\eta^{-1}, \tau\in [0,\theta_\eta(r)], \omega\in\mathbb{S}^{n-1}}.
$$

Now define 
$$
\rho_\eta(r)=\epsilon-\theta_\eta(r).
$$
As $\theta_\eta$ is continuous, so is $\rho_\eta$. Using the coordinates transformation formula between $\Phi$ and $\Psi_{\mathcal{C}_{\mathbf{v},\eta}}$ one obtains
$$
E_{\mathbf{v},\eta}=\set{\Psi_{\mathcal{C}_{\mathbf{v},\eta}}(p,t)\colon p\in \mathcal{C}_{\mathbf{v},\eta}\setminus B_{\sqrt{2n}\eta^{-1}}, t\in [\rho(|\mathbf{x}(p)|),\epsilon]}.
$$
Finally, as $\rho_\eta$ has an explicit formula, the claimed estimate can be checked directly.
\end{proof}

Next we show the following barrier property for $E_{\mathbf{v}, \eta}$:

\begin{lem}\label{CpctMaxLem}
Let $V$ be an integral $n$-varifold in $\mathbb{R}^{n+1}$.  If $V$ has compact support and is $E$-stationary in $E_{\mathbf{v}, \eta}^\circ$, then $\spt(V) \cap E_{\mathbf{v}, \eta}^\circ=\emptyset$.
\end{lem}

\begin{proof}
Consider the $C^1$ vector field, $\mathbf{Z}$, defined by
$$
\mathbf{Z}(\mathbf{x})=\left\{
\begin{array}{cc}
\nabla f^3_{\mathbf{v},\eta}(\mathbf{x}) & \mbox{if $\mathbf{x}\in E_{\mathbf{v},\eta}$} \\
\mathbf{0} & \mbox{otherwise}
\end{array}
\right. .
$$
As $V$ has compact support and $\mathbf{Z}$ is supported in $E_{\mathbf{v}, \eta}$, we may plug $\mathbf{Z}$ into the first variation formula for the functional $E$. The fact that $V$ is $E$-stationary in $E_{\mathbf{v}, \eta}^\circ$ implies that 
\begin{align*}
0 & =\int \left(\mathrm{div}_{S} \mathbf{Z}+ \frac{\mathbf{x}}{2}\cdot \mathbf{Z}\right) e^{\frac{|\mathbf{x}|^2}{4}} \, d\mu_V\\
 &=\int 3f_{\mathbf{v},\eta}\left(2|\nabla_{S} f_{\mathbf{v},\eta}|^2+f^2_{\mathbf{v},\eta}-2f_{\mathbf{v},\eta}(1+\eta^2)|\nabla_{S} (\mathbf{v}\cdot\mathbf{x})|^2\right) e^{\frac{|\mathbf{x}|^2}{4}} \, d(\mu_V\lfloor E_{\mathbf{v},\eta})
\end{align*}
where $S$ is the $\mu_V$-measurable function that maps $\mathbf{x}$ to the generic tangent (hyper)plane of $\mu_V$ at $\mathbf{x}$. By construction $f_{\mathbf{v}, \eta}<0$ in $E_{\mathbf{v}, \eta}^\circ$, and so, in $E_{\mathbf{v}, \eta}^\circ$, $$3f_{\mathbf{v},\eta}\left(2|\nabla_{S} f_{\mathbf{v},\eta}|^2+f^2_{\mathbf{v},\eta}-2f_{\mathbf{v},\eta}(1+\eta^2)|\nabla_{S} (\mathbf{v}\cdot\mathbf{x})|^2\right) \leq 3f_{\mathbf{v},\eta}^3 <0
$$
and so $\mu_V(E_{\mathbf{v}, \eta}^\circ)=0$.  It follows that $\spt(V)\cap E_{\mathbf{v}, \eta}^\circ=\emptyset$.
\end{proof}

As a consequence, we have the following

\begin{lem}\label{NonCpctMaxLem}
Let $\Sigma$ be a $C^2$-asymptotically conical self-expander. If $\mathcal{C}(\Sigma)\cap U_{\mathbf{v}, \eta}=\emptyset$, then $\Sigma \cap E_{\mathbf{v}, \eta}^\circ=\emptyset$.  
\end{lem}

\begin{proof}
The hypotheses ensure that, for all $\beta\in (0,1)$, $\mathrm{cl}(U_{\mathbf{v}, \beta \eta})\cap \mathcal{C}(\Sigma)=\emptyset$. As $\Sigma$ is $C^2$-asymptotic to $\mathcal{C}(\Sigma)$ this means there is an $R_{\beta}>0$ so that $R>R_{\beta}$ implies $\Sigma\cap \partial B_R$ is disjoint from $\mathrm{cl}(U_{\mathbf{v}, \beta \eta})$ and hence also from $E_{\mathbf{v}, \beta \eta}^\circ$. As $\Sigma \cap \bar{B}_R$ is compact and $E$-stationary in $\Real^{n+1}\backslash \left(\Sigma\cap \partial B_R\right)$, it follows from Lemma \ref{CpctMaxLem} that $\Sigma\cap B_R$ is disjoint from $E_{\mathbf{v}, \beta \eta}^\circ$. As $R$ is arbitrary, this means $E_{\mathbf{v}, \beta \eta}^\circ\cap \Sigma =\emptyset$. Hence, using 
$$
E_{\mathbf{v},\eta}^\circ=\bigcup_{\beta\in (0,1)}E_{\mathbf{v},\beta \eta}^\circ,
$$
it follows that $\Sigma \cap E_{\mathbf{v}, \eta}^\circ= \emptyset$.
\end{proof}

\begin{proof}[Proof of Proposition \ref{UniversalBarrierProp}] 
Define
$$
C_r=\bigcup_{p\in \mathcal{L}(\cC)} U_{p,r}
$$
to be the open conical neighborhood of aperture $2\tan^{-1}(r)$ about $\cC$, and let
$$
C_{r}^c=\Real^{n+1}\backslash C_{r}.	
$$
We note that as $\mathcal{L}(\mathcal{C})$ is of class $C^2$ and compact, there is an $\epsilon=\epsilon(\mathcal{C})>0$ so that $\Psi_\mathcal{C}|_{V_\epsilon(\mathcal{C})}\colon V_\epsilon(\mathcal{C})\to C_{\tan(\epsilon)}$ is a $C^1$ diffeomorphism. Choose an $r_0$ so that $0<r_0<\min\set{\tan(\epsilon),1}$. It is straightforward to verify that
$$
C_{r_0}=\Psi_{\cC}(V_{\tan^{-1}(r_0)}(\cC))=\bigcup_{p\in \cC} B_{\frac{r_0}{\sqrt{1+r_0^2}} |\mathbf{x}(p)|}(p)
$$
and, for each $\mathbf{v}\in C_{r_0}^c\cap \mathbb{S}^n$,
$$
U_{\mathbf{v}, r_0} \cap \cC=\emptyset.
$$
	
Now let
$$
\mathcal{B}(\cC)= \bigcup_{\mathbf{v}\in C^c_{r_0}\cap \mathbb{S}^n} E_{\mathbf{v}, r_0}^{\circ}.
$$
This is the union of open sets so is open. As $E_{\mathbf{v},r_0}\subset U_{\mathbf{v}, {r_0}}$ and $U_{\mathbf{v}, {r_0}}\cap \cC=\emptyset$, $\mathcal{B}(\cC)\cap \cC=\emptyset$. Moreover, as $r_0<1$, this construction ensures that
$$
B_{\sqrt{2n}}\cap  E_{\mathbf{v}, r_0}=\emptyset
$$
and so $B_{\sqrt{2n}} \cap \mathcal{B}(\cC)=\emptyset$. As $B_1\subset B_{\sqrt{2n}}$, it follows that Item \eqref{UnivSubsetItem} holds. Item \eqref{UnivCpctBarrItem} follows directly from Lemma \ref{CpctMaxLem} and the definition of $\mathcal{B}(\cC)$. Indeed, if $\spt(V)\cap \mathcal{B}(\cC)\neq \emptyset$, then $\spt(V)\cap E_{\mathbf{v}, r_0}^\circ\neq\emptyset$ for some $\mathbf{v}$ while $V$ is $E$-stationary in $ E_{\mathbf{v}, r_0}^\circ\subset \mathcal{B}(\cC)$.  As $\spt(V)$ is compact this would contradict  Lemma \ref{CpctMaxLem}.   Item \eqref{UnivBarrItem} follows from Lemma \ref{NonCpctMaxLem} in the exact same fashion.
	
We next verify that Item \eqref{UnivTubItem} holds. To see this first observe that if $p\in \Real^{n+1}\backslash (C_{r_0}\cup B_{2n r_0^{-1}})$, then by construction, $p\in \mathcal{B}(\cC)$. As such, if we set $R_0= 2n r_0^{-1}$, then for $R\geq R_0$, if $p\in \Real^{n+1}\backslash (\mathcal{B}(\cC)\cup \bar{B}_R)$, then $p\in C_{r_0}$. Let $(q,t)\in V_{\tan^{-1}(r_0)}(\cC)$ be the pre-image of $p$ under $\Psi_\cC$. Without loss of generality assume the unit normal, $\mathbf{n}_\cC(q)$, on $\cC$ points towards $p$ and let
$$
\mathbf{v}=\frac{\Psi_\cC(q,\tan^{-1}(r_0))}{|\Psi_\cC(q,\tan^{-1}(r_0))|}\in\mathbb{S}^n.
$$
Observe that $C_{r_0}\cap \mathbb{S}^n$ is the $\tan^{-1}(r_0)$-tubular open neighborhood of $\mathcal{L}(\mathcal{C})$ in $\mathbb{S}^n$. Thus, one has $\mathbf{v}\in \partial C_{r_0}$ and as such 
$$
p\notin E_{\mathbf{v}, {r_0}}^\circ.
$$
Using the fact $\mathbf{n}_\cC(q)=\mathbf{n}_{\cC_{\mathbf{v},r_0}}(q)$ and so $\Psi_{\cC}(q,t)=\Psi_{\cC_{\mathbf{v}}, r_0}(q,t)$, it follows from Lemma \ref{BarrierDecayLem} that 
$$
0\leq \sin(t)<4nr_0^{-1} |\mathbf{x}(p)|^{-2}.
$$
Hence, by elementary trigonometry, the distance from $p$ to $\cC$ is less than $N_0 |\mathbf{x}(p)|^{-1}$ where $N_0=4nr_0^{-1}$. In particular, this shows $\mathbb{R}^{n+1}\setminus (\mathcal{B}(\cC)\cup\bar{B}_R)\subset\mathcal{T}_{N_0R^{-1}}(\cC)$.

Finally, up to increasing $R_0$ so that 
$$
N_0<\frac{\tan(\epsilon)}{2\sqrt{1+\tan(\epsilon)^2}} R_0^2,
$$
Item \eqref{UnivTubItem} ensures that $\Real^{n+1}\backslash ( \mathcal{B}(\cC)\cup \bar{B}_{R_0}) \subset C_{\tan(\epsilon)}$. Hence, if 
$$
V'(\cC)=(\Psi_\cC|_{V_\epsilon(\cC)})^{-1}(\mathbb{R}^{n+1}\setminus (\mathcal{B}(\cC)\cup \bar{B}_{R_0})),
$$
then $\Psi_\cC|_{V'(\cC)}$ is a $C^1$ diffeomorphism onto its image.

To conclude the proof we observe that, for $p\in\cC\setminus\bar{B}_{R_0}$, setting 
$$
\mathbf{v}_\pm=\frac{\Psi_\cC(p,\pm\tan^{-1}(r_0))}{|\Psi_\cC(p,\pm\tan^{-1}(r_0))|}\in\mathbb{S}^n
$$
ensures that $p_t=\Psi_\cC(p,t)\notin\mathcal{B}(\cC)$ for $|t|<\tan^{-1}(r_0)$ if and only if $p_t\notin E_{\mathbf{v}_+,r_0}^\circ\cup E_{\mathbf{v}_-,r_0}^\circ$. Indeed, if the latter holds, then Lemma \ref{BarrierDecayLem} implies $|t|\leq \rho_{r_0}(|\mathbf{x}(p)|)$ and so the distance from $p_t$ to $\cC$ is at most $\sin(\rho_{r_0}(|\mathbf{x}(p)|))|\mathbf{x}(p)|$. This implies $p_t\notin E_{\mathbf{v},r_0}^\circ$ for all $\mathbf{v}\in C_{r_0}^c\cap\mathbb{S}^n$ as otherwise, invoking Lemma \ref{BarrierDecayLem} again, one sees the distance from $p_t$ to $\cC$ is strictly larger than $\sin(\rho_{r_0}(|\mathbf{x}(p)|))|\mathbf{x}(p)|$ giving a contradiction. That is, $p_t\notin\mathcal{B}(\cC)$. The other direction is obvious by the fact $\mathbf{v}_\pm\in\partial C_{r_0}$ and the definition of $\mathcal{B}(\cC)$. Hence, setting $\rho_\cC=\rho_{r_0}$ one observes that $V(\cC)=V'(\cC)$ and the result is proved.
\end{proof}

\section{Partial ordering of hypersurfaces asymptotic to a fixed cone} \label{PartialOrderSec}
For $n,k\geq 2$ and $\alpha\in (0,1)$, fix a $C^{k,\alpha}$-regular cone $\cC\subset\mathbb{R}^{n+1}$ and let $\mathcal{H}(\cC)$ be the set of all $C^{k,\alpha}_*$-asymptotically conical $C^{k,\alpha}$-hypersurfaces with asymptotic cone $\cC$ and without any closed connected components. Let $\mathcal{E}(\cC)\subset\mathcal{H}(\cC)$ be the subset consisting of self-expanders and let $\mathcal{E}_{S}(\cC)\subseteq \mathcal{E}(\cC)$ denote the subset of stable self-expanders. 

A pair $(\omega, \sigma)$ consisting of a closed subset $\omega\subset \mathbb{S}^n$ and a smooth, possibly disconnected, hypersurface $\sigma\subset \mathbb{S}^{n}$ is a \emph{boundary link} if $\partial \omega=\sigma$. Here neither $\omega$ nor $\sigma$ are assumed to be connected. If $(\omega, \sigma)$ is a boundary link, then so is $(\mathbb{S}^n\backslash \mathrm{int}(\omega), \sigma)$. For any hypersurface, $\sigma\subset\mathbb{S}^n$, $\sigma$ may be thought of as a closed $(n-1)$-chain with $\mathbb{Z}_2$ coefficients and so one has an associated class $[\sigma]\in H_{n-1}(\mathbb{S}^n;\mathbb{Z}_2)$. As $H_{n-1}(\mathbb{S}^n;\mathbb{Z}_2)=\{0\}$, $\sigma$ is a boundary and so there is an $\omega$ so $(\omega,\sigma)$ is a boundary link. If $(\omega^\prime,\sigma)$ is also a boundary link, then both $\omega$ and $\omega'$ may be thought of as $n$-chains with $\mathbb{Z}_2$ coefficients and, as $\partial (\omega+\omega')=2\sigma=0$, $\omega+\omega'$ is a cycle. Hence, $[\omega+ \omega']\in H_{n}(\mathbb{S}^n;\mathbb{Z}_2)=\mathbb{Z}_2$ is either $0$ or $[\mathbb{S}^n]$. That is, either $\omega^\prime=\omega$ or $\omega'=\mathbb{S}^n\setminus\mathrm{int}(\omega)$.

Given a cone $\cC$ pick $\omega$ so that $(\omega,\mathcal{L}(\cC))$ is a boundary link. This choice induces a canonical unit normal on $\mathcal{L}(\cC)$ (and hence also on $\cC$) -- i.e., by choosing the outward normal to $\omega$. For each $\Sigma\in\mathcal{H}(\cC)$, let $\Omega_-(\Sigma)$ be the open subset of $\mathbb{R}^{n+1}$ so that $\partial\Omega_-(\Sigma)=\Sigma$ and 
$$
\lim_{\rho\to 0^+}\mathrm{cl}\left(\rho\Omega_-(\Sigma)\right)\cap \mathbb{S}^n=\omega \mbox{ as closed sets}.
$$
Such $\Omega_-(\Sigma)$ is well-defined by the hypotheses on $\Sigma$ and the discussions in the previous paragraph. Denote by $\Omega_+(\Sigma)=\mathbb{R}^{n+1}\setminus \mathrm{cl}(\Omega_-(\Sigma))$. We orient $\Sigma$ so that its unit normal points into $\Omega_+(\Sigma)$ and out of $\Omega_-(\Sigma)$.

We introduce a relation on $\mathcal{H}(\cC)$ as follows: If $\Sigma_1, \Sigma_2 \in \mathcal{H}(\cC)$, then
$$
\Sigma_1 \preceq \Sigma_2 \mbox{ provided } \Omega_+(\Sigma_2) \subseteq \Omega_+(\Sigma_1).
$$
Notice $\Sigma\preceq \Sigma$ for any $\Sigma\in \mathcal{H}(\cC)$. The construction ensures that if $\Sigma_1\preceq\Sigma_2$ and $\Sigma_2\preceq\Sigma_3$, then $\Sigma_1\preceq\Sigma_3$. That is, $(\mathcal{H}(\cC), \preceq)$ is a partially ordered set. Clearly, $(\mathcal{E}(\cC),\preceq)$ and $(\mathcal{E}_S(\cC),\preceq)$ are also partially ordered sets.  Recall that an element $x$ of a partially ordered set $(X, \leq) $ is \emph{maximal} if, for all $y\in X$, $x\leq y\Rightarrow x=y$ and is \emph{minimal} if, for all $y\in X$, $y\leq x\Rightarrow x=y$. The element $x$ is the \emph{greatest} element of $(X, \leq)$ if $y\leq x$ for all $y\in X$ and is the \emph{least} element of $(X,\leq)$ if $x\leq y$ for all $y\in X$.   Clearly, the greatest (least) element is the unique maximal (minimal) element.

We use the universal barrier of Section \ref{UniversalBarrierSec} with a minimization procedure sketched by Ilmanen \cite{IlmanenLec} -- see Ding \cite[Theorem 6.3]{Ding} for full details -- to show that $(\mathcal{E}(\cC), \preceq)$ admits a greatest and least element.

\begin{thm}\label{GreatestThm}
For $k\geq 2$ and $\alpha\in (0,1)$, let $\cC$ be a $C^{k,\alpha}$-regular cone in $\Real^{n+1}$ and assume either $2\leq n\leq 6$ or $\lambda[\cC]< \Lambda_n$. There are unique elements $\Gamma_G, \Gamma_L\in \mathcal{E}_S(\cC)$ so that, for all $\Gamma\in \mathcal{E}(\cC)$, $\Gamma_{L}\preceq\Gamma \preceq \Gamma_G$. 
\end{thm}

We will need several auxiliary results to prove this. First a standard regularity result:

\begin{lem}\label{RegLem}
For $k\geq 2$ and $\alpha\in (0,1)$, let $\cC$ be a $C^{k,\alpha}$-regular cone in $\Real^{n+1}$ and assume either $2\leq n\leq 6$ or $\lambda[\cC]< \Lambda_n$. If $V$ is an $E$-stationary integral varifold with tangent cone at infinity equal to $\cC$ and the singular set, $\mathrm{sing}(V)$, has Hausdorff dimension at most $n-7$, then $V=V_\Sigma$ for an element $\Sigma\in\mathcal{E}(\cC)$.
\end{lem}

\begin{proof}
First observe that there is a self-expander $\Sigma\subset\mathbb{R}^{n+1}$ so $V=V_\Sigma$. When $2\leq n\leq 6$ this follows from our hypothesis on the singular set of $V$. When $n\geq 7$, by Huisken's monotonicity formula $\lambda[V]\leq\lambda[\cC]$ and the claim then follows from standard dimension reduction arguments \cite[Theorem 4]{WhiteStrata}, Allard's regularity theorem \cite[Theorem 24.2]{Simon} and the hypothesis that \eqref{Assump1} holds. Next, by \cite[Proposition 3.3]{BWProperness}, $\Sigma$ is $C^{k,\alpha}_*$-asymptotic to $\cC$. That is, $\Sigma\in\mathcal{E}(\cC)$ completing the proof. 
\end{proof}

A key property of the partial order is that there are always elements of $\mathcal{E}_S(\cC)$ that lie above and below any pair of elements of $\mathcal{E}(\cC)$.

\begin{prop}\label{StableOrderProp}
For $k\geq 2$ and $\alpha\in (0,1)$, let $\cC$ be a $C^{k,\alpha}$-regular cone in $\Real^{n+1}$ and assume either $2\leq n\leq 6$ or $\lambda[\cC]< \Lambda_n$. For any two $\Gamma_1,\Gamma_2\in \mathcal{E}(\cC)$ there are $\Gamma_\pm \in \mathcal{E}_S(\cC)$ with $\Gamma_-\preceq \Gamma_i \preceq \Gamma_+$ for $i=1,2$.  Moreover, one of the following three situations occurs:
\begin{enumerate}
\item $\Gamma_-\preceq \Gamma_1\preceq \Gamma_2\preceq \Gamma_+$.
\item $\Gamma_- \preceq \Gamma_2\preceq \Gamma_1\preceq \Gamma_+$.
\item $\Gamma_\pm \neq \Gamma_1$ and $\Gamma_\pm \neq \Gamma_2$.
\end{enumerate}  
\end{prop}

\begin{proof}
Let $\mathcal{B}(\cC)$ be the universal barrier given by Proposition \ref{UniversalBarrierProp} for the cone $\cC$. As $\Gamma_1$ and $\Gamma_2$ are asymptotic to $\cC$ one has, by Item \eqref{UnivBarrItem} of Proposition \ref{UniversalBarrierProp}, that $\Gamma_i\cap \mathcal{B}(\cC)=\emptyset$ for $i=1,2$.  Let $\Sigma_+=\partial (\Omega_+(\Gamma_1)\cap \Omega_+(\Gamma_2))$ and let $\Sigma_-=\partial (\Omega_-(\Gamma_1)\cap \Omega_-(\Gamma_2))$. Notice that $\Sigma_\pm $ are locally given as the graph of Lipschitz functions and are both hypersurfaces away from $\Gamma_1\cap \Gamma_2$.
 
Let $\Psi_{\cC}$ and $\rho_{\cC}$ be given by Proposition \ref{UniversalBarrierProp}. By Item \eqref{UnivRegNbhdItem} of Proposition \ref{UniversalBarrierProp} one has, for $R>2R_0$ large, that there is a sufficiently small $\epsilon_R>0$ so that if
$$
\gamma^R_\pm=\Psi_{\cC}(R\mathcal{L}(\cC), \pm(\rho_{\cC}(R)-\epsilon_R)),
$$
then $\gamma^R_\pm\cap\mathcal{B}(\cC)=\emptyset$ and $\gamma^R_\pm\subset\Omega_\pm(\Gamma_1)\cap\Omega_\pm(\Gamma_2)$. Moreover, each $\gamma^R_\pm$ is, by construction, homologous to $R\mathcal{L}(\cC)$ and hence is null-homologous in $\mathrm{cl}(\Omega_\pm (\Gamma_1))\cap\mathrm{cl}(\Omega_\pm (\Gamma_2))$. As such, one can minimize the expander functional $E$ in the closed set $\mathrm{cl}(\Omega_\pm (\Gamma_1))\cap\mathrm{cl}(\Omega_\pm (\Gamma_2))\cap \bar{B}_{2R}$ to obtain an integral current $\Gamma^R_\pm$ with $\partial\Gamma^R_\pm=\gamma^R_\pm$. Moreover, by \cite[Proposition 6.1 and Theorem 6.2]{WhiteObstaclePrinciple}, the singular set of $\Gamma^R_\pm$ has Hausdorff dimension less than $n-2$. It follows that the associated varifold of $\Gamma^R_\pm$ is $E$-stationary. Hence, by Solomon-White's \cite{SolomonWhite} maximum principle $\Gamma^R_\pm$ is compactly supported in $\bar{B}_{2R}\cap \Omega_\pm(\Gamma_1)\cap\Omega_\pm (\Gamma_2)$, and so Item \eqref{UnivCpctBarrItem} of Proposition \ref{UniversalBarrierProp} implies $\spt(\Gamma^R_\pm)\cap \mathcal{B}(\mathcal{C})=\emptyset$. 

By Item \eqref{UnivRegNbhdItem} of Proposition \ref{UniversalBarrierProp}, $\mathcal{C}\setminus B_{2R_0}$ is a deformation retract of $\mathbb{R}^{n+1}\setminus (\mathcal{B}(\mathcal{C})\cup B_{2R_0})$. Thus, the construction of $\gamma^R_\pm$ ensures that $[\gamma_\pm^R]=[R\mathcal{L}(\cC)]\neq 0$ in $H_{n-2}(\mathbb{R}^{n+1}\setminus (\mathcal{B}(\mathcal{C})\cup B_{2R_0}))$. Hence, as $\gamma_{\pm}^R=\partial \Gamma_{\pm}^R\subset \Real^{n+1}\backslash \mathcal{B}(\cC)$, $\spt(\Gamma^R_\pm)\cap B_{2R_0}\neq\emptyset$ must hold.

Now pick a sequence $R_i\to \infty$, up to passing to a subsequence, the $\Gamma_\pm^{R_i}$ converge, as integral currents, to a $\Gamma_\pm$ supported in $\mathrm{cl}(\Omega_\pm(\Gamma_1))\cap\mathrm{cl}(\Omega_\pm(\Gamma_2))$. As $\spt(\Gamma_{\pm}^{R_i})\cap B_{2R_0}\neq \emptyset$ and the $\Gamma_{\pm}^{R_i}$ are $E$-minimizing in $\bar{B}_{2R_i}\cap \mathrm{cl}(\Omega_\pm(\Gamma_1))\cap\mathrm{cl}(\Omega_\pm(\Gamma_2))$, it follows that $\spt(\Gamma_\pm)\cap B_{2R_0}\neq \emptyset$ and so the limit is non-trivial. Hence, arguing as in Ding \cite[Theorem 6.3]{Ding}, the tangent cone of the associated varifold, $V_{\Gamma_\pm}$, of each $\Gamma_\pm$ at infinity is $\cC$. As $\Gamma_\pm$ is locally $E$-minimizing,  standard regularity theory for area minimizing hypersurfaces \cite[Theorem 37.7]{Simon} gives the singular set of $\Gamma_\pm$ has Hausdorff dimension at most $n-7$. Hence, by our hypotheses, it follows from Lemma \ref{RegLem} that $\Gamma_\pm\in \mathcal{E}_S(\mathcal{C})$.

If $\Gamma_1\preceq \Gamma_2$, then $\Sigma_+=\Gamma_2$ and $\Sigma_-=\Gamma_1$ and the construction ensures $\Gamma_-\preceq \Gamma_1 \preceq \Gamma_2 \preceq \Gamma_+$ and Case (1) holds. Similarly, if $\Gamma_2\preceq \Gamma_1$, then the construction ensures $\Gamma_-\preceq \Gamma_2 \preceq \Gamma_1 \preceq \Gamma_+$ and Case (2) holds. If neither of these cases hold, then the construction still ensures that $\Gamma_-\preceq \Gamma_i \preceq \Gamma_+$ for $i=1,2$, but one cannot have $\Gamma_1=\Gamma_\pm$ or $\Gamma_2=\Gamma_\pm$; i.e., Case (3) holds.
\end{proof}

We also need the following compactness result.

\begin{prop}\label{CpctProp}
For $k\geq 2$ and $\alpha\in (0,1)$, let $\cC$ be a $C^{k,\alpha}$-regular cone in $\Real^{n+1}$ and assume either $2\leq n\leq 6$ or $\lambda[\cC]< \Lambda_n$.
If $\Sigma_i \in\mathcal{E}_{S}(\cC_i)$ and $\mathcal{L}(\Sigma_i)=\mathcal{L}(\cC_i)\to \mathcal{L}(\cC)$ in $C^{k,\alpha}(\mathbb{S}^n)$, then there is a $\Sigma\in  \mathcal{E}_S(\cC)$ so that, up to passing to a subsequence, $\Sigma_i\to \Sigma$ in $C^\infty_{loc}(\Real^{n+1})$.
In particular, the space $\mathcal{E}_S(\cC)$ is (sequentially) compact in $C^{\infty}_{loc}(\Real^{n+1})$.
\end{prop}	

\begin{proof}
If $n\geq 7$, the hypothesis $\lambda[\cC]< \Lambda_n\leq 2$ and \cite[Theorem 1.1 (3)]{BWProperness} imply that, up to passing to a subsequence, the $\Sigma_i$ converge in $C^\infty_{loc}(\Real^{n+1})$ to an element $\Sigma\in \mathcal{E}(\cC)$. The nature of the convergence ensures $\Sigma\in \mathcal{E}_S(\cC)$. 
  
When $2\leq n \leq 6$, observe that, by \cite[Corollary 3.4]{BWProperness}, there is an $\mathcal{R}=\mathcal{R}(\cC)>0$ so that, up to passing to a subsequence, the $\Sigma_i\backslash \bar{B}_{\mathcal{R}}$ converge -- with multiplicity one -- in $C^\infty_{loc}(\Real^{n+1}\backslash \bar{B}_{\mathcal{R}})$ to $\Sigma^\prime$ a self-expander in $\Real^{n+1}\backslash \bar{B}_{\mathcal{R}}$ that is $C^{k,\alpha}_*$-asymptotic to $\cC$.
  
Furthermore, by \cite[Lemma 3.6]{BWProperness} and  \cite[Lemma 3.8]{BWProperness} one has, for any $R>0$ and $i$ sufficiently large,
$$
\mathcal{H}^n(\Sigma_i\cap B_R)\leq M\lambda[\cC] R^n
$$
where $M$ depends only on the dimension. In particular, each $\Sigma_i \cap B_{R}$ is stable and has uniformly bounded volume and so, by standard compactness results for stable hypersurfaces \cite{SSY}, as $2\leq n \leq 6$ one has uniform curvature estimates.  Hence, up to passing to a further subsequence, $\Sigma_i\cap B_{4\mathcal{R}}\to \Sigma^{\prime\prime}$ in $C^\infty_{loc}(B_{4\mathcal{R}})$ -- here the convergence may, in principle, be with multiplicity greater than one. However, $\Sigma^\prime=\Sigma^{\prime\prime}$ in $B_{3\mathcal{R}}\backslash \bar{B}_{2\mathcal{R}}$ and so $\Sigma=\Sigma^\prime\cup \Sigma^{\prime\prime} \in \mathcal{E}(\cC)$. By the construction and the fact that $\Sigma$ has no closed components, $\Sigma_i\to \Sigma$ in $C^\infty_{loc}(\Real^{n+1})$ with multiplicity one. The nature of the convergence ensures $\Sigma\in \mathcal{E}_S(\cC)$.
\end{proof}

\begin{proof}[Proof of Theorem \ref{GreatestThm}]
It is clear that if $\Gamma_G$ and $\Gamma_L$ exist, then they are unique. Notice that $\mathcal{E}(\cC)\neq\emptyset$. Indeed, by a minimization procedure of Ilmanen \cite{IlmanenLec} and Ding \cite[Theorem 6.3]{Ding} and standard regularity theory \cite[Theorem 37.7]{Simon}, there is a locally $E$-minimizing integral $n$-current, $\Gamma^\prime$, with singular set of Hausdorff dimension at most $n-7$ and with tangent cone of the associated varifold of $\Gamma^\prime$ at infinity equal to $\cC$. Thus, by our hypotheses, it follows from Lemma \ref{RegLem} that $\Gamma^\prime\in\mathcal{E}(\cC)$, proving the claim. 

Now let $\mathcal{B}(\cC)$ be the universal barrier associated to $\cC$ given by Proposition \ref{UniversalBarrierProp}. Pick $\Gamma\in \mathcal{E}(\cC)$, let $\Omega_+=\mathcal{B}(\cC)\cap \Omega_+(\Gamma)$ and $\Omega_-=\mathcal{B}(\cC)\cap \Omega_-(\Gamma)$. As $\Gamma\cap \mathcal{B}(\cC)=\emptyset$, by Item \eqref{UnivBarrItem} of Proposition \ref{UniversalBarrierProp}, $\mathcal{B}(\cC)=\Omega_+\cup \Omega_-$. Clearly, this decomposition is independent of the choice $\Gamma\in \mathcal{E}(\cC)$. In particular, for all $\Gamma\in \mathcal{E}(\cC)$, $\Omega_+\subset \Omega_+(\Gamma) \subset \Real^{n+1}\backslash \Omega_- $ and $\Omega_-\subset \Omega_-(\Gamma)\subset \Real^{n+1}\backslash \Omega_+$.

Let $U_+=\bigcup_{\Gamma\in \mathcal{E}(\cC)} \Omega_+(\Gamma)$ and $U_-= \bigcup_{\Gamma\in \mathcal{E}(\cC)} \Omega_-(\Gamma)$. These are both open subsets of $\Real^{n+1}$. Clearly, $\Omega_+\subset U_+\subset \Real^{n+1}\backslash \Omega_-$ and $\Omega_-\subset U_-\subset \Real^{n+1}\backslash \Omega_+$. In particular,  $\partial U_+$ and $\partial U_-$ are both nonempty. We claim that $\partial U_+=\Gamma_L$ and $\partial U_-=\Gamma_G$.
	
To that end, let $M$ denote the number of components of $\mathcal{L}(\cC)$. As $\mathcal{L}(\cC)$ is compact this is a finite positive integer. As $\Gamma$ has no compact connected components, every element of $\Gamma\in \mathcal{E}(\cC)$ has at most $M$ components. Now fix a $p_1\in \partial U_+$ and observe $p_1\not\in \Omega_+(\Sigma)$ for any $\Sigma\in \mathcal{E}(\cC)$. By definition, there are $q_i\to p_1$ and $\Gamma_i\in \mathcal{E}(\cC)$ with $q_i\in \Omega_+(\Gamma_i)$. By Proposition \ref{StableOrderProp}, there are $\Upsilon_i\in \mathcal{E}_S(\cC)$ with $\Upsilon_i\preceq \Gamma_i$ and so, $q_i \in \Omega_+(\Upsilon_i)$. By Proposition \ref{CpctProp}, up to passing to a subsequence, we have $\Upsilon_i\to \Sigma_1\in \mathcal{E}_S(\cC)$. As $q_i \in  \Omega_+(\Upsilon_i)$, while $p_1\not\in  \Omega_+(\Upsilon_i)$, one must have $p_1\in \Sigma_1$. Let  $\Sigma_1^0$ be the component of $\Sigma_1$ containing $p_1$. 

If $\partial U_+=\Sigma_1$, then $\partial U_+\in \mathcal{E}_S(\cC)$. If not, we may pick $p_2\in \partial U_+\backslash \Sigma_1$. By definition, there are $q_i^\prime\to p_2$ and $\Gamma^\prime_i\in \mathcal{E}(\cC)$ with $q_i^\prime\in \Omega_+(\Gamma^\prime_i)$.  Applying Proposition \ref{StableOrderProp} to the pairs $\Sigma_1$ and $\Gamma^\prime_i$, one produces elements $\Upsilon^\prime_i \in \mathcal{E}_S(\cC)$ with $\Upsilon^\prime_i \preceq \Gamma^\prime_i$ and $\Upsilon^\prime_i \preceq \Sigma_1$. By Proposition \ref{CpctProp}, up to passing to a subsequence, we have $\Upsilon^\prime_i\to \Sigma_2\in \mathcal{E}_S(\cC)$. Observe that, as above, one must have $p_2\in \Sigma_2$. The fact that $\Upsilon^\prime_i \preceq \Sigma_1$ implies that $\Sigma_2 \preceq \Sigma_1$ and, hence, that $p_1\in \Sigma_2$.  It follows from the strong maximum principle that $\Sigma_1^0\subset \Sigma_2$. Hence, the component of $\Sigma_2$ containing $p_1$ is equal to $\Sigma_1^0$ and so $p_1$ and $p_2$ are in different components of $\Sigma_2$.

If $\partial U_+=\Sigma_2$, then $\partial U_+\in \mathcal{E}_S(\cC)$. If not, we may pick $p_3\in \partial U_+\backslash \Sigma_2$. Arguing as above produces a $\Sigma_3\in \mathcal{E}_S(\cC)$ with $p_1, p_2$ and $p_3 $ in different components of $\Sigma_3$. As any element of $\mathcal{E}(\cC)$ has at most $M$ components this procedure must stop after $m\leq M$ steps. That is, it produces an element $\Sigma_m\in \mathcal{E}_S(\cC)$ with $\partial U_+=\Sigma_m\in \mathcal{E}_S(\cC)$. 

Hence, we have shown $\partial U_+\in \mathcal{E}_S(\cC)$. By construction, $\partial U_+\preceq \Gamma$ for any $\Gamma\in \mathcal{E}(\cC)$. That is, $\Gamma_L=\partial U_+\in \mathcal{E}_S(\cC)$ is indeed the least element. A similar argument shows that $\partial U_-\in \mathcal{E}_S(\cC)$ and satisfies $\Gamma \preceq \partial U_-$ for all $\Gamma \in \mathcal{E}(\cC)$ and so $\Gamma_G=\partial U_-$ is the greatest element. 
\end{proof}

\section{Expander mean-convex mean curvature flow of low entropy} \label{FlowSec}
Let $\set{\Sigma_t}_{t\in I}$ be a MCF. Along the flow, we define the \emph{expander mean curvature relative to the space-time point $X_0=(\mathbf{x}_0,t_0)$} to be 
$$
E^{X_0}_{\Sigma_t}(p)=2(t-t_0) H_{\Sigma_t}(p)+(\mathbf{x}(p)-\mathbf{x}_0)\cdot\mathbf{n}_{\Sigma_t}(p).
$$
We remark that $E^{X_0}_{\Sigma_t}=-S^{X_0}_{\Sigma_t}$ where $S^{X_0}_{\Sigma_t}$ was introduced in \cite[Section 3]{BWDuke} as the shrinker mean curvature relative to $X_0$. Observe that, due to the dependence on $t$, $E^{X_0}_{\Sigma_t}$ is defined for the flow. For a time $t\in\mathbb{R}$ and a hypersurface $\Sigma$, the expander mean curvature of $\Sigma$ relative to the space-time point $X_0$ and time $t$ is defined to be
$$
E^{X_0,t}_{\Sigma}(p)=2(t-t_0) H_{\Sigma}(p)+(\mathbf{x}(p)-\mathbf{x}_0)\cdot\mathbf{n}_{\Sigma}(p).
$$
Denote by $O$ the space-time origin. For $\beta>0$ we let
$$
\psi_\beta(s)=s^{-\beta} e^{-\beta s} \mbox{ for $s>0$}.
$$
The main result of this section is the following:

\begin{prop} \label{FlowProp}
For $n,k\geq 2$ and $\alpha\in (0,1)$, let $\Sigma\in\mathcal{ACH}^{k,\alpha}_n$ have no closed connected components and let $\Omega$ be an open subset of $\mathbb{R}^{n+1}$ so $\partial\Omega=\Sigma$. Suppose the following holds:
\begin{enumerate}
\item \label{AsympConeHypo} There is an $N>1$ so that $\Sigma\setminus B_{NR}\subset \mathcal{T}_{R^{-1}}(\mathcal{C}(\Sigma))$ for all $R>1$; 
\item \label{RescaleMCHypo} There are constants $c, \beta>0$ so that, by choosing the outward unit normal to $\Omega$, for $p\in\Sigma$,
$$
E^{O,1}_{\Sigma}(p) \geq c \psi_{\beta}(1+|\mathbf{x}(p)|^2)>0;
$$
\item \label{EntropyHypo} $\lambda[\Sigma]<\Lambda_n^*$.
\end{enumerate}
Then there is a unique MCF, $\set{\Sigma_t}_{t\geq 1}$ with $\Sigma_1=\Sigma$ and a family of open subsets of $\mathbb{R}^{n+1}$, $\set{\Omega_t}_{t\geq 1}$ with $\Omega_1=\Omega$ and $\partial\Omega_t=t^{-\frac{1}{2}}\Sigma_t$ so that:
\begin{enumerate}
\item \label{AsympConeItem} Each $\Sigma_t\in\mathcal{ACH}^{k,\alpha}_n$ with $\mathcal{C}(\Sigma_t)=\mathcal{C}(\Sigma)$; 
\item \label{RescaleMCItem} By choosing the outward unit normal to $\Omega_t$, for $t\geq 1$ and $p\in\Sigma_t$,
$$
E^O_{\Sigma_t}(p)\geq c\psi_\beta(1+|\mathbf{x}(p)|^2+2n(t-1))>0;
$$
\item \label{LongtimeItem} For any $1\leq t<t^\prime$, $\mathrm{cl}(\Omega_{t^\prime})\subset\Omega_t$ and so
$$
\lim_{t\to\infty} \mathrm{cl}(\Omega_t)=K \mbox{ as closed sets}
$$
where $\partial K=\Gamma$ is a stable asymptotically conical self-expander with $\mathcal{C}(\Gamma)=\mathcal{C}(\Sigma)$. Moreover,
$$
\lim_{t\to\infty} \partial\Omega_t=\Gamma \mbox{ in $C^\infty_{loc}(\mathbb{R}^{n+1})$}
$$ 
and, hence, $\Gamma$ and $\Sigma$ are $C^{k,\alpha}$ a.c.-isotopic with fixed cone.
\end{enumerate}
\end{prop}

Before the proof of Proposition \ref{FlowProp} we will need several auxiliary lemmas. Fix a unit vector $\mathbf{e}$, a point $\mathbf{x}_0\in\mathbb{R}^{n+1}$ and $r,h>0$. Let 
$$
C_{\mathbf{e}}(\mathbf{x}_0,r,h)=\set{\mathbf{x}\in\mathbb{R}^{n+1}\colon |(\mathbf{x}-\mathbf{x}_0)\cdot\mathbf{e}|<h \mbox{ and } |\mathbf{x}-\mathbf{x}_0|^2<r^2+|(\mathbf{x}-\mathbf{x}_0)\cdot\mathbf{e}|^2}
$$
be the solid open cylinder with axis $\mathbf{e}$ centered at $\mathbf{x}_0$ and of radius $r$ and height $2h$. Recall the following definition from \cite[Section 3]{BWProperness}.

\begin{defn}
Let $k\geq 2$ be an integer and $\alpha\in (0,1)$. A hypersurface $\Sigma\subset\mathbb{R}^{n+1}$ is a $C^{k,\alpha}$ $\mathbf{e}$-graph of size $\delta$ on scale $r$ at $\mathbf{x}_0$ if there is a function $f\colon B^n_r\subset P_{\mathbf{e}}\to\mathbb{R}$ with
$$
\sum_{j=0}^k r^{-1+j} \Vert\nabla^j f\Vert_0+r^{-1+k+\alpha} [\nabla^k f]_\alpha < \delta,
$$
where $P_\mathbf{e}$ is the $n$-dimensional subspace of $\mathbb{R}^{n+1}$ normal to $\mathbf{e}$, so that 
$$
\Sigma\cap C_{\mathbf{e}}(\mathbf{x}_0, r, \delta r)=\set{\mathbf{x}_0+\mathbf{x}(x)+f(x)\mathbf{e}\colon x\in B_r^n}.
$$
\end{defn}

\begin{lem} \label{AsympFlowLem}
For $n,k\geq 2$ and $\alpha\in (0,1)$, let $\set{\Sigma_t}_{t\in [1,T)}$ be a MCF in $\mathbb{R}^{n+1}$. Suppose $\Sigma_1\in\mathcal{ACH}^{k,\alpha}_n$ with the asymptotic cone $\mathcal{C}=\mathcal{C}(\Sigma_1)$ and that there is an $N>0$ so that $\Sigma_1\setminus B_{NR}\subset\mathcal{T}_{R^{-1}}(\mathcal{C})$ for all $R\geq 1$. Given $\gamma\in (0,1)$ there are constants $\tilde{N}=\tilde{N}(\Sigma_1,\gamma)>1$ and $\eta=\eta(\Sigma_1,\gamma)>0$ so that 
\begin{enumerate}
\item \label{AsympFlowConeItem} For all $R\geq 1$ and $t\in [1,T)$, $\Sigma_t\setminus B_{\tilde{N} R \sqrt{t}}\subset \mathcal{T}_{R^{-1}\sqrt{t}}(\mathcal{C})$;
\item \label{AsympFlowGraphItem} For $t\in [1,T)$, $\Sigma_t$ is a $C^{k,\alpha}$ $\mathbf{n}_{\mathcal{C}}(p)$-graph of size $\gamma$ on scale $\eta |\mathbf{x}(p)|$ at every $p\in\mathcal{C}\setminus B_{\tilde{N}\sqrt{t}}$. In particular,
$$
\sup_{t\in [1,T)}\sup_{\Sigma_t\setminus B_{\tilde{N}\sqrt{t}}} |A_{\Sigma_t}|<\infty.
$$
\end{enumerate}
\end{lem}

\begin{proof}
Fix any $t\in [1,T)$ and define $\Gamma_s=t^{-\frac{1}{2}} \Sigma_{1+t(s+1)}$ for $-1\leq s< t^{-1}(T-1)-1$. Thus, $\set{\Gamma_s}$ is a MCF and the hypothesis on $\Sigma_1$ implies $\Gamma_{-1}\in\mathcal{ACH}^{k,\alpha}_n$ and $\Gamma_{-1}\setminus B_{NR}\subset\mathcal{T}_{R^{-1}}(\mathcal{C})$ for all $R\geq 1$. Thus, by \cite[Lemma 4.3]{BWDuke}, there is an $N^\prime=N^\prime(\mathcal{C},N,n)>1$ so that, for all $R\geq 1$ and $s\in [-1,-t^{-1}]$, $\Gamma_s\setminus B_{N^\prime R}\subset\mathcal{T}_{R^{-1}}(\mathcal{C})$. For $s=-t^{-1}$, this gives $\Sigma_t\setminus B_{N^\prime R\sqrt{t}}\subset\mathcal{T}_{R^{-1}\sqrt{t}}(\mathcal{C})$ for all $R\geq 1$, proving Item \eqref{AsympFlowConeItem}.

Let $\delta\in (0,1)$ be a number to be chosen. As $\Sigma_1$ is $C^{k,\alpha}_*$-asymptotically conical, there is an $\epsilon=\epsilon(\Sigma_1,\delta)\in (0,1)$ and $\tilde{N}=\tilde{N}(\Sigma_1,\delta)>4N^\prime$ so that $\Sigma_1$ is a $C^{k,\alpha}$ $\mathbf{n}_{\mathcal{C}}(p)$-graph of size $\delta$ on scale $8r$, where $r=r(p)=\epsilon |\mathbf{x}(p)|$, at every $p\in\mathcal{C}\setminus B_{\tilde{N}}$ and, by scaling, so is $\Gamma_{-1}$. Thus, by the pseudo-locality property for MCF, \cite[Theorem 1.5]{IlmanenNevesSchulze}, one may choose $\delta$ sufficiently small so that, for every $p\in\mathcal{C}(\Sigma_1)\setminus B_{\tilde{N}}$ and $s\in [-1,-t^{-1}]$, $\Gamma_s\cap C_{\mathbf{n}_{\mathcal{C}}(p)}(p,4r,4r)$ is given by the graph of a function $f_p(s,x)$ over (some subset of) $T_{p}\mathcal{C}$ which satisfies
$$
(4r)^{-1}\Vert f_p(s,\cdot)\Vert_{0; B^n_{4r}}+\Vert \nabla f_p(s,\cdot)\Vert_{0; B^n_{4r}} \leq 1
$$
where $\nabla$ is the gradient in spatial variable $x$. As $\set{\Gamma_s}$ is a MCF, $f_p(s,x)$ satisfies
$$
\frac{\partial f_p}{\partial s}=\sqrt{1+|\nabla f_p|^2}\, \mathrm{div}\left(\frac{\nabla f_p}{\sqrt{1+|\nabla f_p|^2}}\right).
$$
It follows from the H\"{o}lder estimates for quasi-linear parabolic equations \cite[Theorem 1.1 of Chapter 6]{LSU} that given $\alpha^\prime\in (0,1)$ there is a $C=C(n,\alpha^\prime)$ so that
$$
\sup_{s\in [-1,-t^{-1}]}[\nabla f_p(s,\cdot)]_{\alpha^\prime; B^n_{2r}} +\sup_{x\in B_{2r}^n} [\nabla f_{p}(\cdot,x)]_{\frac{\alpha^\prime}{2}; [-1,-t^{-1}]} \leq Cr^{-\alpha^\prime}.
$$
Hence, by the Schauder estimates (see, e.g., \cite[Theorem 5.1 of Chapter 4]{LSU}), one has that, for every $s\in [-1,-t^{-1}]$,
\begin{equation}\label{SchauderEqn}
\sum_{j=0}^k r^{j-1} \Vert \nabla^j f_p(s,\cdot)\Vert_{0; B^n_{r}}+r^{k+\alpha-1} [\nabla^k f_p(s,\cdot)]_{\alpha; B^n_{r}} \leq C^\prime
\end{equation}
and that 
$$
\sup_{x\in B^n_{r}} [\nabla f_p(\cdot,x)]_{\frac{1}{2};[-1,-t^{-1}]} \leq C^\prime r^{-1}
$$
where $C^\prime=C^\prime(n,k,\alpha)>1$. 

These estimates together with the equation of $f_p$ implies, for every $s\in [-1,-t^{-1}]$, that
\begin{align*}
|f_p(s,x) & -f_p(-1,0)| \leq |f_p(s,x)-f_p(-1,x)|+|f_p(-1,x)-f_p(-1,0)| \\
& \leq (s+1)\Vert\partial_s f_p(\cdot,x)\Vert_{0; [-1,-t^{-1}]}+|\nabla f_p(-1,0)||x|+|x|^2\Vert \nabla^2 f_p(-1,\cdot)\Vert_{0;B^n_r} \\
& \leq |\nabla f_p(-1,0)| |x|+C^{\prime\prime} r^{-1} \left(|x|^2+(s+1)\right)
\end{align*}
and
\begin{align*}
|\nabla f_p(s,x)-\nabla f_p(-1,0)| & \leq |\nabla f_p(s,x)-\nabla f_p(s,0)|+|\nabla f_p(s,0)-\nabla f_p(-1,0)| \\
& \leq |x| \Vert\nabla^2 f_p(s,\cdot)\Vert_{0;B_r^n}+\sqrt{s+1} \, [\nabla f_p(\cdot,0)]_{\frac{1}{2}; [-1,-t^{-1}]} \\
& \leq C^{\prime\prime}r^{-1}\left(|x|+\sqrt{s+1}\right)
\end{align*}
where $C^{\prime\prime}=C^{\prime\prime}(n,C^\prime)>C^\prime$. Observe that, by Item \eqref{AsympFlowConeItem}, $|f_p(-1,0)|<N^\prime |\mathbf{x}(p)|^{-1}$ and $|\nabla f_p(-1,0)|\leq\delta$. Thus, for any $\rho\in (0,1)$,
\begin{equation} \label{C1Eqn}
(\rho r)^{-1} \Vert f_p(s,\cdot)\Vert_{0; B^n_{\rho r}}+\Vert \nabla f_p(s,\cdot)\Vert_{0; B^n_{\rho r}} \leq 2C^{\prime\prime}(\delta+\rho+N^\prime\epsilon^{-2}\rho^{-1}\tilde{N}^{-1}).
\end{equation}

Hence, combining \eqref{SchauderEqn} and \eqref{C1Eqn} gives, for all $s\in [-1,-t^{-1}]$, 
$$
\sum_{j=0}^k (\rho r)^{j-1} \Vert \nabla^j f_p(s,\cdot)\Vert_{0; B^n_{\rho r}} +(\rho r)^{k+\alpha-1} [\nabla^k f_p(s,\cdot)]_{\alpha; B^n_{\rho r}} \leq  4C^{\prime\prime}(\delta+\rho+N^\prime\epsilon^{-2}\rho^{-1}\tilde{N}^{-1}).
$$
Now choose $\delta=\rho=\frac{\gamma}{16C^{\prime\prime}}$ and enlarge $\tilde{N}$ to ensure that the right side of the above estimate is less than $\gamma$. As $\Gamma_s=t^{-\frac{1}{2}}\Sigma_t$ when $s=-t^{-1}$, Item \eqref{AsympFlowGraphItem} follows immediately from this by setting $\eta=\epsilon \rho$.  
\end{proof}

\begin{lem} \label{RescaleMCLem}
Let $\set{\Sigma_t}_{t\in [1,T)}$ be a MCF in $\mathbb{R}^{n+1}$ and assume $\Sigma_1$ is a $C^2$-hypersurface of finite entropy. If the following holds:
\begin{enumerate}
\item For some $c, \beta>0$, by a suitable choice of the unit normal on $\Sigma_1$, for $p\in\Sigma_1$,
$$
E^{O}_{\Sigma_1}(p) \geq c \psi_\beta(1+|\mathbf{x}(p)|^2);
$$
\item For some $\tilde{N}>0$,
$$
\sup_{\Sigma_1}|A_{\Sigma_1}|+\sup_{t\in [1,T)}\sup_{\Sigma_t\setminus B_{\tilde{N} \sqrt{t}}}  |A_{\Sigma_t}|<\infty,
$$
\end{enumerate}
then, for $t\in [1,T)$ and $p\in\Sigma_t$,
$$
E^{O}_{\Sigma_t}(p)\geq c \psi_\beta(1+|\mathbf{x}(p)|^2+2n(t-1))
$$
where the unit normal on $\Sigma_t$ is chosen to be compatible with the one on $\Sigma_1$.
\end{lem}

\begin{proof}
First of all, by \cite[Proposition 4]{Smoczyk},
$$
\left(\frac{d}{dt}-\Delta_{\Sigma_t}\right) E^O_{\Sigma_t}=|A_{\Sigma_t}|^2 E^{O}_{\Sigma_t}.
$$
Let 
$$
\varrho(p,t)=1+|\mathbf{x}(p)|^2+2n(t-1)
$$
and observe that, by \cite[Lemma 1.1]{EHAnn},
$$
\left(\frac{d}{dt}-\Delta_{\Sigma_t}\right)\varrho=0.
$$
Thus, the chain rule gives
$$
\left(\frac{d}{dt}-\Delta_{\Sigma_t}\right) \psi_{\beta}(\varrho)=-\psi_\beta^{\prime\prime}(\varrho) |\nabla_{\Sigma_t} \varrho|^2 \leq 0
$$
where 
$$
\psi_\beta^{\prime\prime}(s)=\left(\beta(\beta+1)+2\beta^2 s +\beta^2s^2\right) s^{-\beta-2} e^{-\beta s} >0.
$$
Hence, combining the equations for $E^O_{\Sigma_t}$ and $\psi_\beta(\varrho)$ gives
$$
\left(\frac{d}{dt}-\Delta_{\Sigma_t}\right)(c\psi_\beta(\varrho)-E^O_{\Sigma_t}) \leq -|A_{\Sigma_t}|^2 E^O_{\Sigma_t}\leq |A_{\Sigma_t}|^2 (c\psi_\beta(\varrho)-E^O_{\Sigma_t}).
$$

As the flow is regular\footnote{That is, the flow is smooth away from initial time and attains its initial data in the $C^2_{loc}(\mathbb{R}^{n+1})$ topology.} on $[1,T)$, Hypothesis (2) implies that, for all $T_0\in (1,T)$,  there is a constant $C=C(T_0)$ so that for all $t\in [0,T_0]$
$$
\sup_{\Sigma_t} \left(|A_{\Sigma_t}|^2(p)+|\psi_\beta(\varrho(p,t))|+ (1+|\mathbf{x}(p)|)^{-1} |E^O_{\Sigma_t}(p)|\right)\leq C.
$$
That is, $c\psi_\beta(\varrho)-E^O_{\Sigma_t}$ has at most linear growth on $\Sigma_t$ for each $t\in [1,T_0]$ and the second fundamental form is uniformly bounded by $C$. It follows from a non-compact maximum principle (e.g., a simple modification of the proof of \cite[Corollary 1.1]{EHAnn}) and the fact that on $\Sigma_1$,  $c\psi_\beta(\varrho)-E^O_{\Sigma_1}\leq 0$, that $c\psi_\beta(\varrho)-E^O_{\Sigma_t}\leq 0 $ for all $t\in [1,T_0]$.  As $T_0$ was arbitrary in $(1,T)$, the claim follows.
\end{proof}

\begin{lem} \label{BoundCurvLem}
Let $\set{\Sigma_t}_{t\in [1,T)}$ be a MCF in $\mathbb{R}^{n+1}$ and assume $\Sigma_1$ is a $C^2$-hypersurface of finite entropy. If the following holds:
\begin{enumerate}
\item For some $c, \beta>0$, by a suitable choice of the unit normal on $\Sigma_t$, for $t\in [1,T)$ and $p\in\Sigma_t$,
$$
E^{O}_{\Sigma_t}(p) \geq c \psi_\beta(1+|\mathbf{x}(p)|^2+2n(t-1));
$$
\item For some $\tilde{N}>0$,
$$
\tilde{M}=\sup_{\Sigma_1} |A_{\Sigma_1}|+\sup_{t\in [1,T)}\sup_{\Sigma_t\setminus B_{\tilde{N} \sqrt{t}}} |A_{\Sigma_t}|<\infty,
$$
\end{enumerate}
then, for $t\in [1,T)$ and $p\in\Sigma_t$,
$$
\psi_\beta\left(1+|\mathbf{x}(p)|^2+2n(t-1)+\tilde{N}^2t\right)|A_{\Sigma_t}|(p) \leq \tilde{M} c^{-1}  E^O_{\Sigma_t}(p).
$$
\end{lem}

\begin{proof}
On $\Sigma_t\setminus B_{\tilde{N}\sqrt{t}}$, the desired estimate follows from our hypotheses. Next we define
$$
u=|A_{\Sigma_t}|^2v^2=|A_{\Sigma_t}|^2 |E^O_{\Sigma_t}|^{-2}.
$$
By Appendix B, (B.9) of \cite{EckerBook},
$$
\left(\frac{d}{dt}-\Delta_{\Sigma_t}\right)|A_{\Sigma_t}|^2 \leq -2|\nabla_{\Sigma_t}|A_{\Sigma_t}||^2+2|A_{\Sigma_t}|^4.
$$
A direct computation (see, e.g., (3.22)-(3.24) of \cite{BWDuke}) gives
$$
\left(\frac{d}{dt}-\Delta_{\Sigma_t}\right) u \leq -2\nabla_{\Sigma_t} \log v\cdot\nabla_{\Sigma_t} u.
$$
Thus, the maximum principle implies 
$$
\sup_{\mathcal{S}\cap \left(B_{\tilde{N}\sqrt{t}}\times[1,t]\right)}u \leq \sup_{\mathcal{S}\cap\partial_P \left(B_{\tilde{N}\sqrt{t}}\times [1,t]\right)} u
$$
where $\mathcal{S}=\bigcup_{\tau\in [1,T)}\Sigma_\tau\times\set{\tau}$ is the space-time track of the flow and
$$
\partial_P \left(B_{\tilde{N}\sqrt{t}}\times [1,t]\right)=\partial \left(B_{\tilde{N}\sqrt{t}}\times [1,t]\right)\setminus \left(B_{\tilde{N}\sqrt{t}}\times\set{t}\right).
$$
Our hypotheses ensure that
$$
\sup_{\mathcal{S}\cap\partial_P \left(B_{\tilde{N}\sqrt{t}}\times [1,t]\right)} u \leq \frac{\tilde{M}^2}{ c^2 \psi_\beta^{2}\left(1+\tilde{N}^2t+2n(t-1)\right)}.
$$
Hence the desired estimate holds on $\Sigma_t\cap B_{\tilde{N}\sqrt{t}}$ as well.
\end{proof}

With a minor modification of the proof of \cite[Proposition 4.5]{BWDuke}, we use Lemma \ref{BoundCurvLem} to prove the long time existence of certain expander mean-convex MCFs of low entropy.

\begin{prop} \label{LongtimeProp}
For $n\geq 2$, let $\Sigma_1$ be a $C^2$-hypersurface in $\mathbb{R}^{n+1}$ and assume $\Sigma_1$ has no closed connected components and $\lambda[\Sigma_1]<\lambda[\mathbb{S}^{n-1}\times\mathbb{R}]$. Suppose $T\in (0,\infty]$ is the maximal existence time of the MCF $\set{\Sigma_t}_{t\in [1,T)}$ starting from $\Sigma_1$. If the following holds: 
\begin{enumerate}
\item For some $c, \beta>0$, by a suitable choice of the unit normal on $\Sigma_t$, for $t\in [1,T)$ and $p\in\Sigma_t$,
$$
E^{O}_{\Sigma_t}(p) \geq c \psi_\beta(1+|\mathbf{x}(p)|^2+2n(t-1));
$$
\item For some $\tilde{N}>0$,
$$
\tilde{M}=\sup_{\Sigma_1} |A_{\Sigma_1}|+\sup_{t\in [1,T)}\sup_{\Sigma_t\setminus B_{\tilde{N} \sqrt{t}}} |A_{\Sigma_t}|<\infty,
$$
\end{enumerate}
then $T=\infty$.
\end{prop}

\begin{proof}
We argue by contradiction. If $T<\infty$, then Hypothesis (2) implies
$$
\lim_{t\to T} \sup_{\Sigma_t\cap B_{\tilde{N}\sqrt{T}}} |A_{\Sigma_t}|=\infty.
$$
Thus, by Huisken's \cite{Huisken} monotonicity formula and Brakke's \cite{Brakke} (cf. \cite{WhiteReg}) regularity theorem, there is an $\mathbf{x}_0\in\bar{B}_{\tilde{N}\sqrt{T}}$ so that the rescaled MCF about $X_0=(\mathbf{x}_0,T)$,
$$
\Gamma_s=(T-t)^{-\frac{1}{2}} (\Sigma_t-\mathbf{x}_0), \, s=-\log(T-t),
$$
satisfies that, for some sequence $s_i\to\infty$, the $\Gamma_{s_i}$ converge, as integral varifolds, to a multiplicity-one $F$-stationary varifold, $\Gamma$, with $1<\lambda[\Gamma]<\lambda[\mathbb{S}^{n-1}\times\Real]$. 

By the hypotheses and Lemma \ref{BoundCurvLem}, there is a constant $C=C(n,c,\beta,\tilde{N},\tilde{M},T)>0$ so that, for any $p\in\Gamma_{s_i}\cap B_{R_i}$ where $R_i=e^{\frac{s_i}{2}}$,
$$
|A_{\Gamma_{s_i}}|(p) \leq C\left(2(T-e^{-s_i})H_{\Gamma_{s_i}}(p)+e^{-\frac{s_i}{2}} (\mathbf{x}_0+e^{-\frac{s_i}{2}}\mathbf{x}(p))\cdot\mathbf{n}_{\Gamma_{s_i}}(p)\right).
$$
Passing $s_i\to\infty$ and invoking Brakke's regularity theorem again, one has
$$
|A_\Gamma|\leq 2CTH_\Gamma \mbox{ on the regular set $\mathrm{Reg}(\Gamma)$}.
$$
As $\lambda[\mathbb{S}^{n-1}\times\mathbb{R}]<2$, it follows from standard dimension reduction arguments \cite[Theorem 4]{WhiteStrata}, regularity of rectifiable mod $2$ flat chains \cite{WhiteMod2} and Allard's regularity theorem \cite[Theorem 24.4]{Simon}, that $\Gamma$ is regular everywhere -- cf. \cite[Proposition 5.1]{CIMW}.

Hence, $H_\Gamma$ does not change sign and, as $1<\lambda[\Gamma]<\lambda[\mathbb{S}^{n-1}\times\mathbb{R}]$, it follows from \cite[Theorem 0.14]{CMGen} that $\Gamma$ is the self-shrinking sphere. As each $\Gamma_{s_i}$ has no closed connected components, neither does $\Gamma$. Thus, $\Gamma$ cannot be a sphere, giving a contradiction.
\end{proof}

\begin{proof}[Proof of Proposition \ref{FlowProp}]
Fix a transverse section $\mathbf{v}$ on $\Sigma$ so that $\mathbf{v}\in C^{k,\alpha}_0\cap C_{0,\mathrm{H}}^k(\Sigma;\mathbb{S}^{n})$. For $\Sigma\in\mathcal{ACH}^{k,\alpha}_n$, there is an open neighborhood $U$ of $\Sigma$ and a $C^{k,\alpha}$ diffeomorphism $\Phi_{\mathbf{v}}\colon\Sigma\times (-\epsilon,\epsilon)\to U$ given by 
$$
\Phi_\mathbf{v}(p,\tau)=\mathbf{x}(p)+\tau\mathbf{v}(p).
$$
Using $\Phi_\mathbf{v}$, the MCF starting from $\Sigma_1$ can be expressed as a quasi-linear parabolic equation on $\Sigma$ with initial data $0$. Thus, by standard parabolic theory (e.g., \cite{LSU} or \cite{EHInvent}), there is a unique MCF, $\set{\Sigma_t}_{t\in [1,T)}$ with $\Sigma_1=\Sigma$ and $T$ the maximal existence time. As each $\Sigma_t$ is properly embedded in $\Real^{n+1}$,  there is a one-parameter family of open subsets of $\mathbb{R}^{n+1}$, $\set{\Omega_t}_{t\in [1,T)}$ so $\partial \Omega_t=t^{-\frac{1}{2}}\Sigma_t$ and so the outward unit normals $\mathbf{n}_{\Sigma_t}$ to $\Omega_t$ are continuous in $t$. Hence, by the hypotheses on $\Sigma$, one appeals to Lemmas \ref{AsympFlowLem} and \ref{RescaleMCLem} and Proposition \ref{LongtimeProp} to see the flows exists for all times, i.e., $T=\infty$, and Items \eqref{AsympConeItem} and \eqref{RescaleMCItem}. 
 
To analyze the asymptotic behaviors of the flow at $t=\infty$, we define 
$$
\Gamma_s=t^{-\frac{1}{2}}\Sigma_t \mbox{ and } K_s=\mathrm{cl}(\Omega_t) \mbox{ where } s=\log t.
$$
Thus, $\set{\Gamma_s}_{s\geq 0}$ satisfies the rescaled MCF equation
$$
\left(\frac{\partial\mathbf{x}}{\partial s}\right)^\perp=\mathbf{H}_{\Gamma_s}-\frac{\mathbf{x}^\perp}{2}.
$$
Observe, by Item \eqref{RescaleMCItem}, the expander mean curvature of $\Gamma_s$, $\mathbf{H}_{\Gamma_s}-\frac{1}{2}\mathbf{x}^\perp$, points into $K_s$. Thus, $K_{s^\prime}\subset\mathrm{int}(K_s)$ for all $s^\prime>s\geq 0$.

We consider the translation in time of $\set{\Gamma_s}_{s\geq 0}$ by $\tau>0$,
$$
\set{\Gamma_s^\tau}_{s\geq 0}=\set{\Gamma_{s+\tau}}_{s\geq 0},
$$
which is also a rescaled MCF. As $\lambda[\Sigma]<\Lambda^*_n<2$, it follows from Huisken's \cite{Huisken} monotonicity formula and the scaling invariance of entropy that $\lambda[\Gamma^\tau_s]<\Lambda^*_n$. Thus, by Brakke's \cite{Brakke} compactness theorem (see also \cite[Section 7]{IlmanenElliptic}), given a sequence $\tau_{i}\to\infty$ there is a subsequence $\tau_{i_j}$ so that, for every $s\geq 0$, 
\begin{equation} \label{ConvergeEqn}
\lim_{j\to\infty}\mathcal{H}^n\lfloor\Gamma^{\tau_{i_j}}_s=\mu_s 
\end{equation}
where $\set{\mu_s}_{s\geq 0}$ is a one-parameter family of multiplicity-one rectifiable Radon measures satisfying the rescaled MCF equation in Brakke's sense -- see \cite[Section 11]{WhiteStrata} for the precise definition. Moreover, by the monotonicity of $K_s$ and the upper semi-continuity of Gaussian density function, one has for all $s\geq 0$
$$
\mathrm{spt}(\mu_s)=\partial K \mbox{ where $K=\bigcap_{s\geq 0} K_s$}.
$$
In particular, $\mu_s=\mathcal{H}^n\lfloor\partial K$ for all $s\geq 0$, and $\set{\mu_s}_{s\geq 0}$ is a static solution of the rescaled MCF. Consequently, the convergence \eqref{ConvergeEqn} can be taken for all $\tau\to\infty$.

Furthermore, by Huisken's \cite{Huisken} monotonicity formula, all tangent flows of $\set{\mu_s}_{s\geq 0}$ are multiplicity-one static minimal (hyper)cones in $\mathbb{R}^{n+1}$. As $\lambda[\mu_s]<\Lambda_n^*$, it follows from White's stratification theorem, \cite[Theorem 4]{WhiteStrata}, that these minimal cones have at most $(n-3)$-dimensional spines and so the singular set of $\set{\mu_s}_{s\geq 0}$ has parabolic Hausdorff dimension at most $n-1$. As the flow is static, the varifold $V_{\partial K}$ associated to $\partial K$ is a multiplicity-one $E$-stationary varifold and the singular set of $V_{\partial K}$ has Hausdorff dimension at most $n-3$. Moreover, by our previous discussion, $\set{\partial K_s}_{s\geq 0}$ form a foliation of a neighborhood of $\partial K$ in $\mathbb{R}^{n+1}\setminus\mathrm{int}(K)$ so that $\mathbf{H}_{\partial K_s}-\frac{1}{2}\mathbf{x}^\perp$ points into $K_s$. Thus, it follows from the maximum principle of Solomon-White \cite{SolomonWhite} that $V_{\partial K}$ is locally one-sided $E$-minimizing (strictly speaking one should think of $K$ as a set of locally finite perimeter and this set is one-sided locally $E$-minimizing).  As such, the regular part of $V_{\partial K}$ is $E$-stable and so it follows from Schoen-Simon's \cite{SchoenSimon} regularity theorem that the Hausdorff dimension of the singular set of $V_{\partial K}$ is at most $n-7$. Moreover, it follows from Lemma \ref{AsympFlowLem} and the Arzel\`{a}-Ascoli theorem that the tangent cone of $V_{\partial K}$ at infinity is equal to $\mathcal{C}(\Sigma)$. Hence it follows from the entropy bound and Lemma \ref{RegLem} that $\partial K\in\mathcal{E}_S(\mathcal{C}(\Sigma))$. Set $\Gamma=\partial K$. As $\Gamma$ is smooth, Brakke's \cite{Brakke} regularity theorem (see also \cite{WhiteReg}) implies that, as $s\to\infty$, the $\Gamma_s$ converge locally smoothly to $\Gamma$. That is,
$$
\lim_{t\to\infty} t^{-\frac{1}{2}} \Sigma_t=\Gamma \mbox{ in $C^\infty_{loc}(\mathbb{R}^{n+1})$}.
$$
Moreover, it follows from Lemma \ref{AsympFlowLem} and the locally smooth convergence that, for any  transverse section $\mathbf{w}$ on $\Gamma$ with $\mathbf{w}\in C^{k,\alpha}_0\cap C^k_{0,\mathrm{H}}(\Gamma;\mathbb{S}^{n})$, there is a large $s_0>1$ so that if $s>s_0$, then there is a function $w_s\in C^{k,\alpha}_1\cap C^0_{-1}(\Gamma)$ so that if
$$
\mathbf{f}_s(p)=\mathbf{x}|_{\Gamma}(p)+w_s(p) \mathbf{w}(p)
$$
then $\mathbf{f}_s\in \mathcal{ACH}^{k,\alpha}_n(\Gamma)$ is a parametrization of $\Gamma_s$ and
$$
\lim_{s\to \infty} \Vert \mathbf{f}_s-\mathbf{x}|_{\Gamma} \Vert_{1}^{(1)}=0.
$$

Finally, we show that $\Gamma$ and $\Sigma$ are $C^{k,\alpha}$ a.c.-isotopic with fixed cone. First of all, by Lemma \ref{IsotopyLem} and the above observation, there is a value $s_1>s_0$ for which $\Gamma_{s_1}$ is $C^{k,\alpha}$ a.c.-isotopic with fixed cone to $\Gamma$ and in fact $\Gamma_s$ is a.c.-isotopic to $\Gamma$ for all $s>s_1$. It follows from Lemma \ref{AsympFlowLem} and basic continuity properties of the MCF that for any $s\in [1,\infty)$ there is an $\epsilon_s>0$ so, via the path given by Lemma \ref{IsotopyLem}, $\Gamma_{s^\prime}$ is $C^{k,\alpha}$ a.c.-isotopic with fixed cone to $\Gamma_s$ for any $s^\prime\in [1,\infty)\cap (s-\epsilon_s, s+\epsilon_s)$.  As $[1,s_1]$ is compact, this implies that there are a finite set of times $1=s^\prime_0<s^\prime_1<\ldots s^\prime_m=s_1$ so $\Gamma_{s_i^\prime}$ and $\Gamma_{s_{i+1}'}$ are $C^{k,\alpha}$ a.c.-isotopic with fixed cone. Hence, composing these finitely many a.c.-isotopies one concludes that $\Gamma$ and $\Sigma$ are $C^{k,\alpha}$ a.c.-isotopic with fixed cone, finishing the proof.
\end{proof}

\section{Deformation of unstable self-expanders} \label{DeformUnstableExpanderSec}
The main result of this section is the following:

\begin{thm} \label{UnstableExpanderThm}
For $n,k\geq 2$ and $\alpha\in (0,1)$, let $\Gamma\in\mathcal{ACH}^{k,\alpha}_n$ be an unstable self-expander and assume $\lambda[\Gamma]<\Lambda^*_n$. There are stable self-expanders $\Gamma_-$ and $\Gamma_+$ in $\mathcal{ACH}^{k,\alpha}_n$ with $\mathcal{C}(\Gamma_-)=\mathcal{C}(\Gamma_+)=\mathcal{C}(\Gamma)$ so that $\Gamma_- \preceq \Gamma \preceq \Gamma_+$ and both $\Gamma_-$ and $\Gamma_+$ are $C^{k,\alpha}$ a.c.-isotopic with fixed cone to $\Gamma$. Moreover, if $\Gamma^\prime\in\mathcal{ACH}^{k,\alpha}_n$ is a stable self-expander with $\mathcal{C}(\Gamma^\prime)=\mathcal{C}(\Gamma)$ and $\Gamma^\prime\preceq\Gamma$ (respectively, $\Gamma\preceq\Gamma^\prime$), then one may choose $\Gamma_-$ (respectively, $\Gamma_+$) to have the additional property that $\Gamma^\prime\preceq\Gamma_-\preceq\Gamma$ (respectively, $\Gamma\preceq\Gamma_+\preceq\Gamma^\prime$).
\end{thm}

To prove Theorem \ref{UnstableExpanderThm} we will need several auxiliary results. The first is the continuity of entropy under small $C^2_1$ perturbations.

\begin{lem} \label{EntropyContLem}
Given $\Gamma\in\mathcal{ACH}^{2}_n$ and $\delta>0$, there exists a $\tau=\tau(\Gamma,\delta)>0$ so that if $\Vert\mathbf{f}-\mathbf{x}|_\Gamma\Vert^{(1)}_2<\tau$, then $|\lambda[\mathbf{f}(\Gamma)]-\lambda[\Gamma]|<\delta$.
\end{lem}

\begin{proof}
First, by the definition of entropy, there is a $\rho_0>0$ and $\mathbf{y}_0\in\mathbb{R}^{n+1}$ so that
$$
F[\rho_0\Gamma+\mathbf{y}_0]>\lambda[\Gamma]-\frac{\delta}{2}.
$$
Thus, for sufficiently small $\tau$, 
$$
F[\rho_0\mathbf{f}(\Gamma)+\mathbf{y}_0]>F[\rho_0\Gamma+\mathbf{y}_0]-\frac{\delta}{2}.
$$
Hence, combining these gives
$$
\lambda[\mathbf{f}(\Gamma)]\geq F[\rho_0\mathbf{f}(\Gamma)+\mathbf{y}_0]>\lambda[\Gamma]-\delta.
$$

Observe that there is a $C>0$ so that, for sufficiently small $\tau$, 
$$
\sup_{\mathbf{y}\in\mathbb{R}^{n+1},R>0} R^{-n}\mathcal{H}^n(\mathbf{f}(\Gamma)\cap B_R(\mathbf{y}))<C.
$$ 
Moreover, given $\gamma,\eta>0$, there is an $\mathcal{R}=\mathcal{R}(\Gamma,\gamma,\eta)>0$ so that if $p\in\rho\mathbf{f}(\Gamma)$ and $\rho+|\mathbf{x}(p)|>\mathcal{R}$, then $\rho\mathbf{f}(\Gamma)$ is an $\mathbf{n}_{\rho\mathbf{f}(\Gamma)}(p)$-graph of size $\gamma$ on scale $\eta$ at $p$. Thus, for a suitable choice of $\gamma$ and $\eta$ depending on $C$ and $\delta$,
\begin{equation} \label{NonCpctRegionEqn}
\sup_{\rho+|\mathbf{y}|>\mathcal{R}} F[\rho\mathbf{f}(\Gamma)+\mathbf{y}]<1+\delta\leq\lambda[\Gamma]+\delta.
\end{equation}
Next we observe that given $\mathcal{R}^\prime>\epsilon^\prime>0$ there are constants $C^\prime=C^\prime(\Gamma,\mathcal{R}^\prime,\epsilon^\prime)$ and $\rho^\prime=\rho^\prime(\Gamma,\mathcal{R}^\prime,\epsilon^\prime)>0$ so that, for $\rho<\rho^\prime$,
$$
\Vert\rho\mathbf{f}(\rho^{-1}\cdot)-\rho\mathbf{x}(\rho^{-1}\cdot)\Vert_{2; \rho\Gamma\cap (B_{2\mathcal{R}^\prime}\setminus B_{\epsilon^\prime})}<C^\prime\tau.
$$
Now choose $\epsilon^\prime$ sufficiently small and $\mathcal{R}^\prime>\mathcal{R}$ sufficiently large so 
$$
(4\pi)^{-\frac{n}{2}} C\left(\mathcal{L}^n\left(B^n_{\epsilon^\prime}\right)+ \int_{\mathbb{R}^n\setminus B^n_{\mathcal{R}^\prime}} e^{-\frac{|\mathbf{x}|^2}{4}} \, d\mathcal{L}^n\right)<\frac{\delta}{4},
$$
where $\mathcal{L}^n$ is the Lebesgue measure on $\Real^n$. As $\lim_{\rho\to 0}\rho\Gamma=\mathcal{C}(\Gamma)$ in $C^2_{loc}(\mathbb{R}^{n+1}\setminus\set{\mathbf{0}})$, it follows that for sufficiently small $\tau$
\begin{equation} \label{SmallRhoRegionEqn}
\sup_{0<\rho<\rho^\prime, \mathbf{y}\in B_{\mathcal{R}^\prime}} F[\rho\mathbf{f}(\Gamma)+\mathbf{y}]<\lambda[\mathcal{C}(\Gamma)]+\delta\leq\lambda[\Gamma]+\delta
\end{equation}
where the second inequality is implied from the lower semi-continuity of entropy. Finally, as the set 
$$
K^\prime=\set{(\rho,\mathbf{y})\colon \rho^\prime\leq\rho\leq\mathcal{R}, \mathbf{y}\in\bar{B}_{\mathcal{R}^\prime}}
$$
is compact, by shrinking $\tau$ if needed, we get
\begin{equation} \label{CpctRegionEqn}
\sup_{(\rho,\mathbf{y})\in K^\prime} F[\rho\mathbf{f}(\Gamma)+\mathbf{y}]<\sup_{(\rho,\mathbf{y})\in K^\prime} F[\rho\Gamma+\mathbf{y}]+\delta\leq \lambda[\Gamma]+\delta.
\end{equation}
Hence, combining \eqref{NonCpctRegionEqn}-\eqref{CpctRegionEqn} gives, for sufficiently small $\tau$, 
$$
\lambda[\mathbf{f}(\Gamma)]<\lambda[\Gamma]+\delta,
$$
which completes the proof.
\end{proof}

The next is an asymptotic estimate of the distance between two disjoint self-expanders that are asymptotic to the same cone.

\begin{prop} \label{InitialBarrierProp}
For $n\geq 2$, let $\Gamma$ and $\Sigma$ be self-expanders in $\mathcal{ACH}^{2}_n$ such that $\mathcal{C}(\Gamma)=\mathcal{C}(\Sigma)$ and $\Gamma\cap\Sigma=\emptyset$. There is a radius $R_1=R_1(\Gamma,\Sigma)>1$ and a constant $C_1=C_1(\Gamma,\Sigma)>0$ so that there is a smooth function $u\colon\Gamma\setminus\bar{B}_{R_1}\to\mathbb{R}^{n+1}$ that satisfies
$$
\Sigma\setminus\bar{B}_{2R_1} \subset \set{\mathbf{x}(p)+u(p)\mathbf{n}_\Gamma(p)\colon p\in\Gamma\setminus\bar{B}_{R_1}} \subset \Sigma
$$
and
$$
C_1^{-1} r^{-n-1} e^{-\frac{r^2}{4}} \leq u \leq C_1 r^{-n-1} e^{-\frac{r^2}{4}}
$$
where $\mathbf{n}_\Gamma$ is chosen to point towards $\Sigma$ and $r(p)=|\mathbf{x}(p)|$ for $p\in\Gamma$.
\end{prop}

\begin{proof}
Appealing to \cite[Proposition 2.1]{BWExpanderRelEnt} one has all claims but the lower bound of $u$. To see this lower bound, a standard computation (see, e.g., \cite[Lemma A.2]{BWExpanderRelEnt}) gives that, up to increasing $R_1$, 
$$
L_\Gamma u=\Delta_\Gamma u+\frac{1}{2}\mathbf{x}\cdot\nabla_\Gamma u+\left(|A_\Gamma|^2-\frac{1}{2}\right) u=\mathbf{a}\cdot\nabla_\Gamma u+bu
$$
where $\mathbf{a}$ and $b$ satisfy
$$
|\mathbf{a}|+|b| \leq C\left(|u|+|\nabla_\Gamma u|+|\nabla^2_\Gamma u|+|\mathbf{x}\cdot\nabla_\Gamma u|\right).
$$
Invoking the gradient and Hessian estimate for $u$ given by \cite[Proposition 2.1]{BWExpanderRelEnt}, one has $|\mathbf{a}|+|b|$ decays with a rate at least $e^{-\frac{r^2}{8}}$. Thus, set $v=(r^{-n-1}+r^{-n-2}) e^{-\frac{r^2}{4}}$ and, up to further increasing $R_1$, one readily computes (see \cite[Lemma A.1]{BWExpanderRelEnt})
$$
L_\Gamma v \geq \mathbf{a}\cdot\nabla_\Gamma v+bv \mbox{ on $\Gamma\setminus\bar{B}_{R_1}$}.
$$
As $\Gamma\cap\Sigma=\emptyset$, there is a constant $C_1>0$ so that
$$
\inf_{\Gamma\cap\partial B_{R_1}} u \geq C_1^{-1} (R_1^{-n-1}+R_1^{-n-2}) e^{-\frac{R_1^2}{4}}.
$$
As $\Gamma$ is $C^2$-asymptotically conical, we can enlarge $R_1$ so 
$$
\sup_{\Gamma\setminus\bar{B}_{R_1}} \left(|A_\Gamma|^2-\frac{1}{2}-b\right)<-\frac{1}{4}.
$$
Hence the maximum principle implies
$$
u \geq C_1^{-1} v >C_1^{-1} r^{-n-1} e^{-\frac{r^2}{4}}
$$
proving the claim.
\end{proof}

\begin{proof}[Proof of Theorem \ref{UnstableExpanderThm}]
Partition $\Gamma$ into its connected components $\Gamma^1,\ldots,\Gamma^M$. For $\Gamma$ to be unstable, at least one $\Gamma^j$ must also be unstable. Invoking \cite[Proposition 3.2]{BWMinMax}, for each $\Gamma^j$ there is a number $\mu_j$ and a function $f_j>0$ on $\Gamma^j$ so that 
$$
(L_{\Gamma^j}+\mu_j) f_j=0
$$
and $f_j$ satisfies the pointwise estimates
$$
\frac{1}{C_0^\prime} (1+|\mathbf{x}|^2)^{-\frac{1}{2}(n+1-2\mu_j)} e^{-\frac{|\mathbf{x}|^2}{4}}\leq f_j \leq C_0^\prime  (1+|\mathbf{x}|^2)^{-\frac{1}{2}(n+1-2\mu_j)} e^{-\frac{|\mathbf{x}|^2}{4}}
$$
and, for every $i\geq 1$,
$$
\Vert e^{\frac{|\mathbf{x}|^2}{8}} \nabla^i_{\Gamma} f_j \Vert_0 \leq C_i^\prime
$$
where $C_0^\prime=C_0^\prime(\Gamma^j)>0$ and $C^\prime_i=C_i^\prime(\Gamma^j)$. Moreover, if $\Gamma^j$ is unstable, then $\mu_j<0$. 

Define $f\colon\Gamma\to\mathbb{R}$ by 
$$
f(p)=\left\{\begin{array}{cc} f_j(p) & \mbox{if $p\in\Gamma^j$ for some unstable $\Gamma^j$}; \\ 0 & \mbox{otherwise}. \end{array}\right.
$$
Given a number $\epsilon$, let 
$$
\Gamma^\epsilon=\set{\mathbf{f}^\epsilon(p)=\mathbf{x}(p)+\epsilon f(p)\mathbf{n}_{\Gamma}(p)\colon p\in\Gamma}
$$
where $\mathbf{n}_\Gamma$ is chosen to point out of $\Omega_-(\Gamma)$. The estimates of $f_j$ ensure that there is a sufficiently small $\bar{\epsilon}=\bar{\epsilon}(\Gamma)>0$ so that, for all $|\epsilon|<\bar{\epsilon}$, $\Gamma^\epsilon\in\mathcal{ACH}^{k,\alpha}_n$ with $\mathcal{C}(\Gamma^\epsilon)=\mathcal{C}(\Gamma)$.

We wish to apply Proposition \ref{FlowProp} to $\Gamma^{\epsilon,j}=\mathbf{f}^\epsilon(\Gamma^j)$ for those unstable $\Gamma^j$. Suppose $-\bar{\epsilon}<\epsilon<0$. By \eqref{CurvExpanderEqn} and the flow equation, the distance between $\Gamma^j$ and $\mathcal{C}(\Gamma^j)$ decays linearly. This fact and the $C^0$ estimate of $f_j$ ensure that Hypothesis \eqref{AsympConeHypo} of Proposition \ref{FlowProp} is satisfied. Next, choose the unit normal $\mathbf{n}_{\Gamma^\epsilon}$ to point out of $\Omega_-(\Gamma^\epsilon)$ and this induces the choice of normal on each $\Gamma^{\epsilon,j}$. If $\Gamma^j$ is unstable, then one appeals to the computations in \cite[Lemma A.2]{BWExpanderRelEnt} and the estimates of $f_j$ to see that, for all $\epsilon<0$ sufficiently close to $0$,  
$$
E^{O,1}_{\Gamma^{\epsilon,j}} \geq 2(C_0^\prime)^{-1} \epsilon \mu_j \psi_{\beta_j}(1+|\mathbf{x}|^2)
$$
for $\beta_j=\frac{1}{2}(n+1-2\mu_j)>0$. That is, Hypothesis \eqref{RescaleMCHypo} of Proposition \ref{FlowProp} is satisfied. Finally, by Lemma \ref{EntropyContLem} and shrinking $\bar{\epsilon}$ (if needed), one has $\lambda[\Gamma^{\epsilon,j}]\leq \lambda[\Gamma^\epsilon]<\Lambda_n^*$. As such, Hypothesis \eqref{EntropyHypo} of Proposition \ref{FlowProp} is satisfied. Thus, if $\Gamma^j$ is unstable, then one finds an open set $\Omega^{\epsilon,j}$ so that $\partial\Omega^{\epsilon,j}=\Gamma^{\epsilon,j}$ and $\mathbf{n}_{\Gamma^{\epsilon,j}}$ points out of $\Omega^{\epsilon,j}$, and applies Proposition \ref{FlowProp} to $(\Gamma^{\epsilon,j},\Omega^{\epsilon,j})$. If $\Gamma^j$ is stable, then $\Gamma^{\epsilon,j}=\Gamma^j$ and evolves by mean curvature in the self-expanding way. Hence, by our discussions and the maximum principle, there is a unique MCF, $\set{\Gamma^\epsilon_t}_{t\geq 1}$ with $\Gamma^\epsilon_1=\Gamma^\epsilon$ so that each $\Gamma^\epsilon_t\in\mathcal{ACH}^{k,\alpha}_n$ with $\mathcal{C}(\Gamma^\epsilon_t)=\mathcal{C}(\Gamma)$, and the family $\set{t^{-1/2}\Gamma^\epsilon_t}_{t\geq 1}$ evolves, in a strictly monotone manner, into $\Omega_-(\Gamma^\epsilon)$ so 
$$
\lim_{t\to\infty} t^{-\frac{1}{2}}\Gamma^\epsilon_t=\Sigma^\epsilon \mbox{ in $C^\infty_{loc}(\mathbb{R}^{n+1})$}
$$
where $\Sigma^\epsilon\in\mathcal{ACH}^{k,\alpha}_n$ is a stable self-expander with $\mathcal{C}(\Sigma^\epsilon)=\mathcal{C}(\Gamma)$. It follows that $\Sigma^\epsilon\preceq\Gamma^\epsilon\preceq\Gamma$ and $\Sigma^\epsilon$ is $C^{k,\alpha}$ a.c.-isotopic with fixed cone to $\Gamma$. Similar arguments apply to the case $0<\epsilon<\bar{\epsilon}$ and produce a $\Sigma^\epsilon$ with the same properties as above but $\Gamma\preceq\Sigma^\epsilon$.

It remains only to prove the last claim (``Moreover, if..."). Without loss of generality, it suffices to consider the case that $\Gamma^\prime\preceq\Gamma$ but $\Gamma^\prime\neq\Gamma$. Let $\Gamma^\prime_t=\sqrt{t} \, \Gamma^\prime$ for $t>0$. We will show, by choosing $\bar{\epsilon}$ sufficiently small, which may depend on $\Gamma^\prime$ as well, that if $-\bar{\epsilon}<\epsilon<0$, then $\Gamma^\prime_t\preceq\Gamma^\epsilon_t$ for all $t\geq 1$. This would imply $\Gamma^\prime\preceq\Sigma^\epsilon$, proving the claim. Let $\hat{\Gamma}=\Gamma\cap\Gamma^\prime$ and $\hat{\Gamma}_t=\sqrt{t} \, \hat{\Gamma}$ for $t>0$. By the strong maximum principle $\hat{\Gamma}$ is the union of the common components of $\Gamma$ and $\Gamma^\prime$. As $\Gamma^\prime$ is stable so is $\hat{\Gamma}$. Thus, by our construction, $\mathbf{f}^\epsilon(\hat{\Gamma})=\hat{\Gamma}$ and so evolves by mean curvature in the self-expanding manner as well. Thus it is enough to show that, for any $\epsilon\in (-\bar{\epsilon},0)$ fixed, the set
$$
S=\set{s\geq 1\colon \mathrm{cl}(\Omega_-(\Gamma_t^\prime))\setminus\hat{\Gamma}_t\subset \Omega_{-}(\Gamma^\epsilon_t) \mbox{ for all $t\in [1,s]$}}
$$
is equal to $[1,\infty)$.

To see this, first observe that the $C^0$ estimate of $f_j$ and Proposition \ref{InitialBarrierProp} ensure that, for any $\epsilon<0$ very close to $0$, $\mathrm{cl}(\Omega_-(\Gamma^\prime))\setminus\hat{\Gamma}\subset\Omega_-(\Gamma^\epsilon)$ and so $S \neq \emptyset$. Next, if $s_i\in S$ and $s_i\to s$, then $\Omega_-(\Gamma_s^\prime)\subset \Omega_{-}(\Gamma^\epsilon_s)$ and $\mathrm{cl}(\Omega_-(\Gamma_{t}^\prime))\setminus\hat{\Gamma}_t\subset \Omega_{-}(\Gamma^\epsilon_{t})$ for all $t\in [1,s)$. If $s\not\in S$, then  $(\Gamma_{s}^\prime\setminus\hat{\Gamma}_s)\cap (\Gamma^\epsilon_{s}\setminus\hat{\Gamma}_s)\neq \emptyset$ which violates the strong maximum principle. This shows $S$ is closed.

Finally, fix any $s\in S$. As $\Gamma^\epsilon\in\mathcal{ACH}^{k,\alpha}_n$ and the distance to its asymptotic cone decays linearly, one uses Lemma \ref{AsympFlowLem} to find a radius $R>1$ and a family of functions $u(\cdot,t)$ on $\Gamma_t^\prime\setminus\bar{B}_R\to\mathbb{R}$ with uniform $C^2$ bound and so that, for any $t\in [1,s+1]$, 
$$
\Gamma_t^\epsilon\setminus\bar{B}_{2R} \subset\set{\mathbf{x}(p)+u(p,t)\mathbf{n}_{\Gamma_t^\prime}(p)\colon p\in\Gamma_t^\prime\setminus\bar{B}_R} \subset\Gamma_t^\epsilon.
$$
A straightforward, but tedious, computation gives that 
$$
\frac{du}{dt}-\Delta_{\Gamma^\prime_t}u=\mathbf{a}^\prime\cdot\nabla_{\Gamma^\prime_t} u+b^\prime u
$$
where $\mathbf{a}^\prime$ and $b^\prime$ are smooth bounded. As $s\in S$,
$$
\inf_{p\in \Gamma_s^\prime\setminus\bar{B}_{2R}} u(p,s)\geq 0.
$$ 
By the continuity of the flows and the definition of $S$, there is a $\delta>0$ so that 
$$
\inf\set{u(p,t)\colon p\in (\Gamma^\prime_t\setminus\hat{\Gamma}_t)\cap \partial B_{2R}, t\in [s,s+\delta]}>0.
$$
Thus, by a non-compact form of the strong maximum principle, $u(p,t)>0$ for $p\in \Gamma^\prime_t\setminus (\hat{\Gamma}_t\cup B_{2R})$ and $t\in (s,s+\delta]$. This, together with the strong maximum principle on compact regions, implies that $\mathrm{cl}(\Omega_-(\Gamma_t^\prime))\setminus\hat{\Gamma}_t\subset\Omega_-(\Gamma_t^\epsilon)$ for all $t\in [s,s+\delta]$. That is, $S$ is open. Hence, as $[1,\infty)$ is connected, one has $S=[1,\infty)$. This completes the proof.
\end{proof}

\section{Perturbation properties of weakly stable self-expanders}\label{PerturbStableExpanderSec}
We need several perturbation results for weakly stable self-expanders. Specifically, we will need to show that it is possible to connect, via an a.c.-isotopy that does not move the asymptotic cones much along the path, any weakly stable self-expander to a self-expander whose asymptotic cone is a generic perturbation of the initial cone. These results rely on the analysis carried out in \cite{BWBanach}.

We introduce the following notation: Let $n,k\geq 2$ and $\alpha\in (0,1)$. Given a $C^{k,\alpha}$-regular cone $\cC\subset\mathbb{R}^{n+1}$ and a map $\varphi\in C^{k,\alpha}(\mathcal{L}(\cC);\Real^{n+1})$, let 
$$
\mathcal{C}[\varphi]=\set{\rho\varphi(p)\colon p\in \mathcal{L}(\cC), \rho>0}.
$$
Clearly, $\mathcal{C}[\varphi]$ is a set-theoretic cone. As $\mathcal{L}(\cC)$ is compact, there is a neighborhood $\mathcal{V}_{\mathrm{emb}}(\cC)$ of $\mathbf{x}|_{\mathcal{L}(\cC)}$ in $C^{k,\alpha}(\mathcal{L}(\cC); \Real^{n+1})$ so that, for any $\varphi\in\mathcal{V}_{\mathrm{emb}}(\cC)$, $\mathcal{C}[\varphi]$ is a $C^{k,\alpha}$- regular cone and $\mathscr{E}^{\mathrm{H}}_1[\varphi]\colon\cC\to \cC[\varphi]$ is an embedding.

The compactness of $\mathcal{E}_S(\cC)$ together with results of \cite{BWBanach} gives the local finiteness for diffeomorphism types. First we need the following elementary fact:

\begin{lem}\label{ParamLem}
For $n,k\geq 2$ and $\alpha\in (0,1)$, let $\Sigma_i\in\mathcal{ACH}^{k,\alpha}_n$ be self-expanders and suppose $\Sigma_i\to\Sigma$ in $C^\infty_{loc}(\Real^{n+1})$. Let $\sigma$ be a $C^{k,\alpha}$-hypersurface in $\mathbb{S}^n$ and let $\varphi_i\in C^{k,\alpha}(\sigma;\mathbb{R}^{n+1})$ such that $\cC[\varphi_i]=\cC(\Sigma_i)$ and $\varphi_i\to\mathbf{x}|_\sigma$ in $C^{k,\alpha}(\mathbb{S}^n; \mathbb{R}^{n+1})$. Then one has $\Sigma\in \mathcal{E}(\mathcal{C}[\sigma])$ and, for sufficiently large $i$, there are $\mathbf{f}_i\in \mathcal{ACH}_n^{k,\alpha}(\Sigma)$ with $\mathbf{f}_i(\Sigma)=\Sigma_i$ and $\mathrm{tr}_\infty^1[\mathbf{f}_i]=\varphi_i$.
\end{lem}

\begin{proof}
First observe that as each $\Sigma_i$ satisfies \eqref{ExpanderEqn}, the nature of the convergence ensures that $\Sigma$ does as well. By our hypotheses on $\varphi_i$, one has $\mathcal{L}(\Sigma_i)\to\sigma$ in $C^{k,\alpha}(\mathbb{S}^n)$ and so, by \cite[Corollary 3.4]{BWProperness}, $\Sigma\in\mathcal{ACH}^{k,\alpha}_n$ with $\mathcal{C}(\Sigma)=\mathcal{C}[\sigma]$. That is $\Sigma\in \mathcal{E}(\mathcal{C}[\sigma])$. 

Let $\mathbf{h}_i\in C^{k,\alpha}_1\cap C_{1,\mathrm{H}}^{k}(\Sigma;\mathbb{R}^{n+1})$ be chosen to satisfy
$$
\Delta_\Sigma\mathbf{h}_i+\frac{1}{2}\mathbf{x}\cdot\nabla_\Sigma\mathbf{h}_i-\frac{1}{2}\mathbf{h}_i=\mathbf{0} \mbox{ and } \mathrm{tr}_\infty^1[\mathbf{h}_i]={\varphi}_i-\mathbf{x}|_{\sigma}.
$$
By \cite[Corollary 5.8]{BWBanach}, there is a unique such $\mathbf{h}_i$ which satisfies the estimate
$$
\Vert\mathbf{h}_i\Vert_{k,\alpha}^{(1)} \leq C\Vert\varphi_i-\mathbf{x}|_{\sigma}\Vert_{k,\alpha}
$$
where $C$ depends only on $\Sigma$. We then let 
$$
\mathbf{g}_i=\mathbf{x}|_\Sigma+\mathbf{h}_i \mbox{ and } \Upsilon_i=\mathbf{g}_i(\Sigma).
$$
It is clear that, for sufficiently large $i$, $\mathbf{g}_i\in\mathcal{ACH}^{k,\alpha}_n(\Sigma)$ and $\mathrm{tr}_\infty^1[\mathbf{g}_i]={\varphi}_i$. Thus, by  \cite[Proposition 3.3]{BWBanach}, $\Upsilon_i\in\mathcal{ACH}^{k,\alpha}_n$ and $\mathcal{C}(\Upsilon_i)=\mathcal{C}(\Sigma_i)$.

Pick a transverse section $\mathbf{v}$ on $\Sigma$ so that $\mathbf{v}\in C^{k,\alpha}_0\cap C^k_{0,\mathrm{H}}(\Sigma;\mathbb{S}^{n})$. Let $\mathbf{v}_i=\mathbf{v}\circ\mathbf{g}_i^{-1}$ and let $\pi_{\mathbf{v}_i}$ be the projection along $\mathbf{v}_i$ onto $\Upsilon_i$. By \cite[Proposition 3.3]{BWProperness}, for sufficiently large $i$, $\pi_{\mathbf{v}_i}|_{\Sigma_i}\colon\Sigma_i\to\Upsilon_i$ is an element of $\mathcal{ACH}^{k,\alpha}_n(\Sigma_i)$. Thus, there is a unique function $u_i\in C^{k,\alpha}_1\cap C^{k}_{1,0}(\Sigma)$ so that $\Sigma_i$ can be parametrized by the map 
$$
\mathbf{f}_i=(\pi_{\mathbf{v}_i}|_{\Sigma_i})^{-1}\circ\mathbf{g}_i=\mathbf{g}_i+u_i\mathbf{v}
$$
which, by \cite[Proposition 3.3]{BWBanach}, is an element of $\mathcal{ACH}^{k,\alpha}_n(\Sigma)$ and $\mathrm{tr}_\infty^1[\mathbf{f}_i]=\mathrm{tr}_\infty^1[\mathbf{g}_i]=\varphi_i$. This completes the proof.
\end{proof}

\begin{prop}\label{FinitenessProp}
For $k\geq 2$ and $\alpha\in (0,1)$, let $\cC$ be a $C^{k,\alpha}$-regular cone in $\Real^{n+1}$ and assume either $2\leq n \leq 6$ or $\lambda[\cC]<{\Lambda_n}$. There is an $\epsilon_1=\epsilon_1(\cC)>0$ and a finite set $\set{\Gamma_1, \ldots, \Gamma_J}\subseteq\mathcal{E}_{S}(\cC)$ so that the following is true: For any $\varphi\in C^{k,\alpha}(\mathcal{L}(\cC); \Real^{n+1})$ with $\Vert \varphi-\mathbf{x}|_{\mathcal{L}(\cC)}\Vert_{k,\alpha}<\epsilon_1$ and any $\Gamma\in\mathcal{E}_S(\mathcal{\cC}[\varphi])$, there is an integer $i \in [1, J]$ and an element $\mathbf{f}\in \mathcal{ACH}_{n}^{k,\alpha}(\Gamma_i)$ so that $\Gamma=\mathbf{f}(\Gamma_i)$ and $\mathrm{tr}^1_\infty[\mathbf{f}]=\varphi$. 
\end{prop}	

\begin{proof}
We first claim that there are $\Gamma_1, \ldots, \Gamma_J\in\mathcal{E}_S(\cC)$ so that for any $\Gamma\in\mathcal{E}_S(\cC)$ there is an integer $i \in [1,J]$ and an element $\mathbf{f}_\Gamma\in \mathcal{ACH}_n^{k,\alpha}(\Gamma_i)$ so that $\mathbf{f}_\Gamma(\Gamma_i)=\Gamma$ and $\mathrm{tr}_\infty^1[\mathbf{f}_\Gamma]=\mathbf{x}|_{\mathcal{L}(\cC)}$.  
	
To see this is true, consider the following equivalence relation on $\mathcal{E}_S(\cC)$: two $\Gamma, \Gamma^\prime\in \mathcal{E}_S(\cC)$ are equivalent, written $\Gamma\sim \Gamma^\prime$, provided there is an $\mathbf{f}\in \mathcal{ACH}_{n}^{k,\alpha}(\Gamma)$ so that $\Gamma^\prime=\mathbf{f}(\Gamma)$ and $\mathrm{tr}_\infty^1[\mathbf{f}]=\mathbf{x}|_{\mathcal{L}(\cC)}$. It follows from \cite[Proposition 3.3]{BWBanach} that this is an equivalence relation. Indeed, it is reflexive as $\mathbf{x}|_{\Gamma}\in \mathcal{ACH}_{n}^{k,\alpha}(\Gamma)$ so $\Gamma\sim \Gamma$. It is symmetric as $\mathbf{f}\in \mathcal{ACH}_{n}^{k,\alpha}(\Gamma)$ with $\mathbf{f}(\Gamma)=\Gamma^\prime$ and $\mathrm{tr}_\infty^1[\mathbf{f}]=\mathbf{x}|_{\mathcal{L}(\cC)}$, implies that $\mathbf{f}^{-1} \in \mathcal{ACH}_{n}^{k,\alpha}(\Gamma^\prime)$ and $\mathrm{tr}_\infty^1[\mathbf{f}^{-1}]=\mathbf{x}|_{\mathcal{L}(\cC)}$. Finally, it is transitive as $\mathbf{f}\in \mathcal{ACH}_{n}^{k,\alpha}(\Gamma)$ with $\mathbf{f}(\Gamma)=\Gamma^\prime$ and $\mathrm{tr}_\infty^1[\mathbf{f}]=\mathbf{x}|_{\mathcal{L}(\cC)}$, and $\mathbf{g}\in \mathcal{ACH}_{n}^{k,\alpha}(\Gamma^\prime)$ with $\mathbf{g}(\Gamma^\prime)=\Gamma^{\prime\prime}$ and $\mathrm{tr}_\infty^1[\mathbf{g}]=\mathbf{x}|_{\mathcal{L}(\cC)}$, implies that $\mathbf{g}\circ \mathbf{f}\in \mathcal{ACH}_{n}^{k,\alpha}(\Gamma)$ and $\mathrm{tr}_\infty^1[\mathbf{g}\circ \mathbf{f}]=\mathbf{x}|_{\mathcal{L}(\cC)}$ so $\Gamma\sim \Gamma^{\prime\prime}$. It readily follows from Proposition \ref{CpctProp} and Lemma \ref{ParamLem} that  there are finitely many equivalence classes in $\mathcal{E}_S(\cC)$. Pick representatives $\Gamma_1, \ldots, \Gamma_J$ and observe that we have shown the proposition for any $\Gamma\in \mathcal{E}_S(\cC)$. 
	
We now argue by contradiction. Suppose there is a sequence of $\varphi_j\in C^{k,\alpha}(\mathcal{L}(\cC); \Real^{n+1})$ with $\Vert \varphi_j-\mathbf{x}|_{\mathcal{L}(\cC)}\Vert_{k,\alpha}\to 0$ and $\Sigma_j \in \mathcal{E}_S(\mathcal{C}[\varphi_j])$ so that the conclusion does not hold for $\Sigma_j$.  By Proposition \ref{CpctProp}, up to passing to a subsequence, there is a $\Sigma \in \mathcal{E}_S(\cC)$ so that $\Sigma_j\to \Sigma$ in $C^\infty_{loc}(\Real^{n+1})$.  By Lemma \ref{ParamLem}, up to throwing out a finite number of terms, there are ${\mathbf{g}}_j\in\mathcal{ACH}_n^{k,\alpha}(\Sigma)$ so that $\mathbf{g}_j(\Sigma)=\Sigma_j$ and $\mathrm{tr}_\infty^1[\mathbf{g}_j]=\varphi_j$.

As $\Sigma\in \mathcal{E}_S(\cC)$, there is an integer $i \in [1,J]$ so $\Gamma_i\sim \Sigma$. That is, there is an $\mathbf{h}\in \mathcal{ACH}_n^{k,\alpha}(\Sigma)$ with $\mathbf{h}(\Gamma_i)=\Sigma$ and $\mathrm{tr}_{\infty}^1[\mathbf{h}]=\mathbf{x}|_{\mathcal{L}(\cC)}$. Setting $\mathbf{f}_j=\mathbf{g}_j\circ \mathbf{h}$, shows the result holds for the $\Sigma_j$, and this contradiction proves the claim.
\end{proof}

Given an element $\Sigma\in\mathcal{ACH}^{k,\alpha}_n$, there is a natural equivalence relation on $\mathcal{ACH}^{k,\alpha}_n(\Sigma)$. Namely, two elements $\mathbf{f},\mathbf{g}\in\mathcal{ACH}^{k,\alpha}_n(\Sigma)$ are equivalent, written $\mathbf{f}\sim\mathbf{g}$, provided that $\mathbf{f}(\Sigma)=\mathbf{g}(\Sigma)$ and $\mathrm{tr}_\infty^1[\mathbf{f}]=\mathrm{tr}_\infty^1[\mathbf{g}]$. Denote by $[\mathbf{f}]$ the equivalence class of $\mathbf{f}$. Let 
$$
\mathcal{ACE}^{k,\alpha}_n(\Sigma)=\set{[\mathbf{f}]\colon \mbox{$\mathbf{f}\in\mathcal{ACH}^{k,\alpha}_n(\Sigma)$ and $\mathbf{f}(\Sigma)$ satisfies \eqref{ExpanderEqn}}}.
$$
The main result of \cite{BWBanach} is that $\mathcal{ACE}^{k,\alpha}_n(\Sigma)$ is a smooth Banach manifold and the projection map $\Pi_{\Sigma}\colon\mathcal{ACE}^{k,\alpha}_n(\Sigma)\to C^{k,\alpha}(\mathcal{L}(\Sigma);\mathbb{R}^{n+1})$ given by $\Pi_{\Sigma}([\mathbf{f}])=\mathrm{tr}_\infty^1[\mathbf{f}]$ is smooth Fredholm of index $0$.

\begin{cor}\label{StrictStableCor}
For $k\geq 2$ and $\alpha\in (0,1)$, let $\cC$ be a $C^{k,\alpha}$-regular cone in $\Real^{n+1}$ and assume either $2\leq n \leq 6$ or $\lambda[\cC]<\Lambda_n$. There is an open neighborhood $\mathcal{V}$  of $\mathbf{x}|_{\mathcal{L}(\cC)}$ in $C^{k,\alpha}(\mathcal{L}(\cC); \Real^{n+1})$ so that, for a generic (in the sense of Baire category) element $\varphi\in\mathcal{V}$, every element of $\mathcal{E}_S(\cC[\varphi])$ is strictly stable.
\end{cor}

\begin{proof}
Pick $\epsilon_1=\epsilon_1(\cC)>0$ as in Proposition \ref{FinitenessProp}. When $n\geq 7$, as $\lambda[\cC]<{\Lambda_n}$, \cite[Lemma 3.8]{BWProperness} ensures that, up to shrinking $\epsilon_1$, if $\Vert\varphi-\mathbf{x}|_{\mathcal{L}(\cC)}\Vert_{k,\alpha}<\epsilon_1$, then $\cC[\varphi]$ satisfies $\lambda[\cC[\varphi]]<{\Lambda_n}$. Let $\mathcal{V}_0$ be the open ball in $C^{k,\alpha}(\mathcal{L}(\cC); \Real^{n+1})$ centered at $\mathbf{x}|_{\mathcal{L}(\cC)}$ of radius $\epsilon_1$. Pick $\Gamma_1, \ldots, \Gamma_J$ as in Proposition \ref{FinitenessProp}. By \cite[Corollary 1.2]{BWBanach}, there are open dense sets $\mathcal{V}_1, \ldots, \mathcal{V}_J\subset C^{k,\alpha}(\mathcal{L}(\cC); \Real^{n+1})$ so that each $\Pi_{\Gamma_i}$  has no critical values in $\mathcal{V}_i$.  That is, if $\Gamma=\mathbf{f}(\Gamma_i)$ and $\mathrm{tr}_\infty^1[\mathbf{f}]\in \mathcal{V}_i$, then $\Gamma$ has no non-trivial Jacobi fields that fix the infinity. In particular, if $\varphi\in \mathcal{V}=\bigcap_{i=0}^J\mathcal{V}_i$, then every element of $\mathcal{E}_S(\cC[\varphi])$ is strictly stable. 
\end{proof}

\begin{thm} \label{StableIsotopyThm}
For $n,k\geq 2$ and $\alpha\in (0,1)$, let $\mathcal{C}$ be a $C^{k,\alpha}$-regular cone in $\mathbb{R}^{n+1}$ and assume $\lambda[\mathcal{C}]<\Lambda_n^*$. For each $\Gamma\in\mathcal{E}_S(\mathcal{C})$ there exists an open neighborhood $\mathcal{V}_\Gamma\subset C^{k,\alpha}(\mathcal{L}(\Gamma);\mathbb{R}^{n+1})$ of $\mathbf{x}|_{\mathcal{L}(\Gamma)}$ so that for any $\varphi\in \mathcal{V}_\Gamma$ there is an element $\Gamma_\varphi\in\mathcal{E}_S(\cC[\varphi])$ and a $C^{k,\alpha}$ a.c.-isotopy $\mathbf{F}_\varphi$ between $\Gamma$ and $\Gamma_\varphi$ so that, for all $t\in [0,1]$,
\begin{equation} \label{TraceDiffEqn}
\Vert\mathrm{tr}_\infty^1[\mathbf{F}_\varphi(t)]-\mathbf{x}|_{\mathcal{L}(\Gamma)}\Vert_{k,\alpha} \leq \Vert\varphi-\mathbf{x}|_{\mathcal{L}(\Gamma)}\Vert_{k,\alpha}.
\end{equation}
\end{thm}

\begin{proof}
Let $\mathbf{v}$ be a transverse section on $\Gamma$ as given in \cite[Section 7]{BWBanach} and let
$$
\mathcal{K}_{\mathbf{v}}=\set{\kappa\in C^2_{loc}\cap C^0_{1,0}(\Gamma)\colon L_\Gamma (\kappa\mathbf{v}\cdot\mathbf{n}_{\Gamma})=0}
$$
where $L_\Gamma$ is the self-joint operator given in Section \ref{IndexSubsec}. By \cite[Theorem 7.1]{BWBanach}, there are two open neighborhoods $\mathcal{U}_1\subset C^{k,\alpha}(\mathcal{L}(\Gamma);\mathbb{R}^{n+1})$ of $\mathbf{x}|_{\mathcal{L}(\Gamma)}$ and $\mathcal{U}_2\subset\mathcal{K}_\mathbf{v}$ of $0$ together with a smooth map $F_\mathbf{v}\colon\mathcal{U}_1\times\mathcal{U}_2\to\mathcal{ACH}_n^{k,\alpha}(\Gamma)$ so that:
\begin{itemize}
\item $F_\mathbf{v}[\mathbf{x}|_{\mathcal{L}(\Gamma)},0]=\mathbf{x}|_\Gamma$;
\item $\mathrm{tr}_\infty^1[F_\mathbf{v}[\varphi,\kappa]]=\varphi$;
\item $G_\mathbf{v}[\varphi,\kappa]=\mathbf{v}\cdot\left(\mathbf{H}-\frac{1}{2}\mathbf{x}^\perp\right)[F_\mathbf{v}[\varphi,\kappa]]\in\mathcal{K}_\mathbf{v}$.
\end{itemize}
Thus, by shrinking $\mathcal{U}_1$ if needed, it follows from Lemma \ref{IsotopyLem} that, for every $\varphi\in\mathcal{U}_1$, the path $t\mapsto \mathbf{F}^\prime_\varphi(t)$ given by
$$
\mathbf{F}^\prime_\varphi(t)=(1-t)\mathbf{x}|_\Gamma+t F_\mathbf{v}[\varphi,0] \mbox{ for $t\in [0,1]$}
$$
is a $C^{k,\alpha}$ a.c.-isotopy between $\Gamma$ and $F_\mathbf{v}[\varphi,0](\Gamma)$. It is clear from the construction that
$$
\Vert\mathrm{tr}_\infty^1[\mathbf{F}_\varphi^\prime(t)]-\mathbf{x}|_{\mathcal{L}(\Gamma)}\Vert_{k,\alpha}\leq \Vert\varphi-\mathbf{x}|_{\mathcal{L}(\Gamma)}\Vert_{k,\alpha}.
$$

To conclude the proof we show that for every $\varphi\in\mathcal{U}_1$ there is a $C^{k,\alpha}$ a.c.-isotopy with fixed cone between $\Sigma_\varphi=F_{\mathbf{v}}[\varphi,0](\Gamma)$ and some element of $\mathcal{E}_S(\mathcal{C}[\varphi])$. Composing this with $\mathbf{F}_\varphi^\prime$ gives an a.c.-isotopy with the desired properties. In view of Theorem \ref{UnstableExpanderThm}, it suffices to show the claim  with $\mathcal{E}(\mathcal{C}[\varphi])$ replacing $\mathcal{E}_S(\mathcal{C}[\varphi])$. If $\mathcal{K}_{\mathbf{v}}=\set{0}$, then $\Sigma_\varphi$ is a $C^{k,\alpha}_*$-asymptotically conical self-expander with $\mathcal{C}(\Sigma_\varphi)=\mathcal{C}[\varphi]$ and the claim is proved. Otherwise, partition $\Gamma$ into its connected components $\Gamma^1,\ldots,\Gamma^M$, so all $\Gamma^j$ are stable with at least one weakly stable. For each $j\in \set{1,\dots,M}$, let $\Sigma_\varphi^j=F_{\mathbf{v}}[\varphi,0](\Gamma^j)$ and 
$$
\mathcal{K}_\mathbf{v}^j=\set{\kappa\in C^2_{loc}\cap C^0_{1,0}(\Gamma^j)\colon L_{\Gamma^j}(\kappa\mathbf{v}\cdot\mathbf{n}_{\Gamma^j})=0}.
$$
Observe that if $\kappa\in\mathcal{K}_{\mathbf{v}}$, then $\kappa|_{\Gamma^j}\in\mathcal{K}_{\mathbf{v}}^j$. If $\Gamma^j$ is strictly stable, then $\mathcal{K}_{\mathbf{v}}^j=\set{0}$ and $\Sigma_\varphi^j$ is a self-expander. If $\Gamma^j$ is weakly stable, then it follows from the standard spectral theory that $\dim\mathcal{K}_{\mathbf{v}}^j=1$. In this case we may choose $\kappa_j\in \mathcal{K}_\mathbf{v}^j$ to span $\mathcal{K}_{\mathbf{v}}^j$. Let $g^j_\mathbf{v}\colon \mathcal{U}_1\to\mathbb{R}$ so that $G_\mathbf{v}[\varphi,0]|_{\Gamma^j}=g^j_\mathbf{v}[\varphi]\kappa_j$. If $g^j_\mathbf{v}[\varphi]=0$, then $\Sigma_\varphi^j$ is a self-expander. Otherwise, it follows from \cite[Lemma 6.1]{BWBanach} and \cite[Proposition 3.2]{BWExpanderRelEnt} that, by a suitable choice of unit normal on $\Sigma_\varphi^j$,
$$
E^{O,1}_{\Sigma^j_\varphi} \geq c\psi_\beta\left(1+|\mathbf{x}|^2\right)
$$
for some $c,\beta>0$. Thus, Hypothesis (2) of Proposition \ref{FlowProp} holds for $\Sigma_\varphi^j$. Next observe that by the construction of $F_\mathbf{v}$ -- see pages 33-34 of \cite{BWBanach} for details -- Hypothesis \eqref{AsympConeHypo} of  Proposition \ref{FlowProp} holds for $\Sigma^j_\varphi$. Finally, by Lemma \ref{EntropyContLem} and shrinking $\mathcal{U}_1$ if needed, we may assume $\lambda[\Sigma_\varphi^j]\leq \lambda[\Sigma_\varphi]<\Lambda_n^*$, that is Hypothesis (3) of Proposition \ref{FlowProp}. Hence, one can apply Proposition \ref{FlowProp} to $\Sigma^j_\varphi$ and obtains a $C^{k,\alpha}_*$-asymptotically conical stable self-expander $\Gamma^j_\varphi$ that is $C^{k,\alpha}$ a.c.-isotopic with fixed cone to $\Sigma^j_\varphi$. Therefore, by the maximum principle for the MCF, combining all these cases gives an element $\Gamma_\varphi\in \mathcal{E}_S(\mathcal{C}[\varphi])$ that is $C^{k,\alpha}$ a.c.-isotopic with fixed cone to $\Sigma_\varphi$.
\end{proof}

\section{Proof of main theorems} \label{ProofMainThmSec}
In this section we prove Theorem \ref{MainThm} and Theorem \ref{AuxThm}. We first prove the result for cones $\cC$ with the extra property that every element of $\mathcal{E}_S(\cC)$ is strictly stable and then use the perturbation results of Section \ref{PerturbStableExpanderSec} to conclude the general case. 

Before beginning the proof we need a finiteness result for $\mathcal{E}_S(\cC)$.

\begin{lem}\label{FiniteLem}
For $k\geq 2$ and $\alpha\in (0,1)$, let $\cC$ be a $C^{k,\alpha}$-regular cone in $\Real^{n+1}$ and assume either $2\leq n \leq 6$ or $\lambda[\cC]<\Lambda_n$. If every element of $\mathcal{E}_S(\cC)$ is strictly stable, then $\mathcal{E}_S(\cC)$ is a finite set.
\end{lem}

\begin{proof}
By Proposition \ref{FinitenessProp} there are $\Gamma_1, \ldots, \Gamma_J\in \mathcal{E}_S(\cC)$ so that for each $\Gamma\in \mathcal{E}_S(\cC)$ there is an integer $i\in [1,J]$ and an element $\mathbf{f}_\Gamma\in\mathcal{ACE}^{k,\alpha}_n(\Gamma_i)$ so that $\mathrm{tr}^1_\infty[\mathbf{f}_\Gamma]= \mathbf{x}|_{\mathcal{L}(\cC)}$ and $\mathbf{f}_\Gamma(\Gamma_i)=\Gamma$. As $\Gamma$ is strictly stable, $[\mathbf{f}_\Gamma]$ is a regular point of $\Pi_{\Gamma}$ -- See Section \ref{PerturbStableExpanderSec} -- and so there is an open neighborhood $\mathcal{U}_\Gamma \subset \mathcal{ACE}^{2,\alpha}_n(\Gamma_i)$ of $[\mathbf{f}_\Gamma]$ on which $\Pi_\Gamma$ restricts to a diffeomorphism. Clearly, $\set{\mathcal{U}_\Gamma}_{\Gamma\in\mathcal{E}_S(\cC)}$ is an open cover of $\set{[\mathbf{f}_\Gamma]}_{\Gamma\in\mathcal{E}_S(\cC)}$. Moreover, 
$$
\set{[\mathbf{f}_\Gamma]}_{\Gamma\in \mathcal{E}_S(\cC)}\cap \mathcal{U}_{\Gamma^\prime}=\set{[\mathbf{f}_{\Gamma^\prime}]}.
$$
	
By Proposition \ref{CpctProp}, $\mathcal{E}_S(\cC)$ is (sequentially) compact in $C^\infty_{loc}(\Real^{n+1})$. Hence, by \cite[Proposition 4.1]{BWProperness}, $\set{[\mathbf{f}_\Gamma]}_{\Gamma\in\mathcal{E}_S(\cC)}$ is (sequentially) compact in $\bigcup_{i=1}^M \mathcal{ACE}_n^{k,\alpha}(\Gamma_i)$. It follows (see \cite[Lemma A.1]{BWProperness}) that $\set{ \mathcal{U}_\Gamma}_{\Gamma\in\mathcal{E}_S(\cC)}$ has a finite subcover of $\set{[\mathbf{f}_\Gamma]}_{\Gamma\in \mathcal{E}_S(\cC)}$ and hence the latter set is finite. That is, $\mathcal{E}_S(\cC)$ is finite.
\end{proof}

We are now ready to prove Theorems \ref{MainThm} and \ref{AuxThm}. 

\begin{proof}[Proof of Theorems \ref{MainThm} and \ref{AuxThm}]
As $\mathcal{C}$ is a $C^{k+1}$-regular cone, it is obvious that $\cC$ is $C^{k,\alpha}$-regular for every $\alpha\in (0,1)$. Observe that, by Theorem \ref{UnstableExpanderThm}, every $\Gamma\in \mathcal{E}(\cC)$ is $C^{k}$ a.c.-isotopic with fixed cone to some element $\Gamma^\prime\in \mathcal{E}_S(\cC)$. Hence, it suffices to show that any two elements $\Gamma_1, \Gamma_2 \in \mathcal{E}_S(\cC)$ are $C^k$ a.c-isotopic with fixed cone. 

We now assume that $\cC$ has the property that $\mathcal{E}_S(\cC)$ consists only of strictly stable elements. For any $\Gamma\in \mathcal{E}_S(\cC)$, let 
$$
\mathcal{P}(\Gamma)=\set{\tilde{\Gamma}\in \mathcal{E}_S(\cC)\colon \Gamma\preceq\tilde{\Gamma}\preceq \Gamma_G}
$$
where $\Gamma_G$ is the greatest element given by Theorem \ref{GreatestThm}. By Lemma \ref{FiniteLem}, $\mathcal{P}(\Gamma)$ is a finite set. We will show, by induction on the number of elements, $M$, of $\mathcal{P}(\Gamma)$, that $\Gamma$ is $C^k$ a.c.-isotopic with fixed cone to $\Gamma_G$. The theorem clearly follows from this. 

To that end, first observe that as $\Gamma_G\in\mathcal{P}(\Gamma)$ there is nothing to prove when $M=1$. For general $M\geq 2$, let $\Gamma^\prime$ be a minimal element of $\mathcal{P}(\Gamma)\setminus\set{\Gamma}$. Thus, $\mathcal{P}(\Gamma^\prime)$ has at most $M-1$ elements. By the induction hypotheses, $\Gamma^\prime$ is $C^k$ a.c.-isotopic with fixed cone to $\Gamma_G$. Apply Theorem \ref{MinMaxThm} to $\Gamma$ and $\Gamma^\prime$. This produces a self-expander $\Sigma$ so that $\Gamma\preceq \Sigma \preceq \Gamma^\prime$ but $\Sigma\neq \Gamma$ and $\Sigma \neq \Gamma'$. As  $\Gamma^\prime$ is a minimal element of $\mathcal{P}(\Gamma)\backslash \set{\Gamma}$ we must have $\Sigma \not\in \mathcal{P}(\Gamma)$ and so $\Sigma \not\in \mathcal{E}_S(\cC)$. In particular, Theorem \ref{UnstableExpanderThm} implies $\Sigma$ is $C^k$ a.c.-isotopic with fixed cone to both $\Gamma$ and $\Gamma^\prime$, and hence they are both $C^{k}$ a.c.-isotopic with fixed cone to one another and hence also to $\Gamma_G$. This completes the proof in this case.

To prove the result for general $\cC$ first pick $\mathcal{V}$ as in Corollary \ref{StrictStableCor}. For any two elements $\Gamma_1, \Gamma_2\in \mathcal{E}_S(\cC)$, let $\mathcal{V}_{\Gamma_1}$ and $\mathcal{V}_{\Gamma_2}$ be given by Theorem \ref{StableIsotopyThm}.  As $\mathcal{V}\cap \mathcal{V}_{\Gamma_1}\cap \mathcal{V}_{\Gamma_2}$ is an open neighborhood of $\mathbf{x}|_{\mathcal{L}(\cC)}$, it follows from Corollary \ref{StrictStableCor} that there is an element $\varphi\in \mathcal{V}\cap \mathcal{V}_{\Gamma_1}\cap \mathcal{V}_{\Gamma_2}$ so that every element of $\mathcal{E}_S(\mathcal{C}[\varphi])$ is strictly stable. By Theorem \ref{StableIsotopyThm}, for $i\in\set{1,2}$ there is an element $\Gamma_i^\prime\in\mathcal{E}_S(\mathcal{C}[\varphi])$ and a $C^{k}$ a.c.-isotopy $\mathbf{F}_i$ between $\Gamma_i$ and $\Gamma_i^\prime$ so that $\eqref{TraceDiffEqn}$ holds. By what we have already shown, $\Gamma_1^\prime$ is $C^k$ a.c.-isotopic with fixed cone to $\Gamma_2^\prime$. Thus, by shrinking $\mathcal{V}_{\Gamma_1},\mathcal{V}_{\Gamma_2}$ if needed, we are able to use Lemma \ref{IsotopyFixConeLem} to conclude that $\Gamma_1$ is $C^k$ a.c.-isotopic with fixed cone to $\Gamma_2$ which completes the proof.
\end{proof}

\appendix

\section{Existence of isotopically trivial self-expanders of small entropy} \label{IsotopyTrivialApp}
In this section we use Theorem \ref{StableIsotopyThm} to prove the following existence result which was used in the proof of Corollary \ref{Uniq2Cor}.

\begin{prop}
For $3\leq n\leq 6$ and $k\geq 2$, if $\cC$ is a $C^{k+1}$-regular cone in $\mathbb{R}^{n+1}$ with $\mathcal{L}(\cC)\in \mathcal{S}^{k+1}_0(\Lambda_n^*)$, then there exists a self-expander $\Gamma$ asymptotic to $\cC$ that is $C^k$ a.c.-isotopic to $\Real^{n}\times\set{0}$.
\end{prop}

\begin{proof}
Let $\Gamma_0=\mathbb{R}^{n}\times\set{0}$ and set $\cC_0=\mathcal{C}(\Gamma_0)=\Gamma_0$. Fix any $\alpha\in (0,1)$. Let 
$$
\mathcal{V}=\set{\varphi\in C^{k,\alpha}(\mathcal{L}(\cC_0);\mathbb{R}^{n+1})\colon \mbox{$\mathscr{E}^{\mathrm{H}}_1[\varphi]$ is an embedding and $\lambda[\mathcal{C}[\varphi]]<\Lambda_n^*$}}
$$
and let $\mathcal{V}_0$ be the connected component of $\mathcal{V}$ that contains $\mathbf{x}|_{\mathcal{L}(\cC_0)}$. As $\mathcal{L}(\cC)\in\mathcal{S}_0^{k+1}(\Lambda_n^*)$ it follows that $\mathbf{x}|_{\mathcal{L}(\cC)}\in\mathcal{V}_0$.

Hence, there is a continuous path $\phi\colon[0,1]\to \mathcal{V}_0$ connecting $\mathbf{x}|_{\mathcal{L}(\cC_0)}$ to $\mathbf{x}|_{\mathcal{L}(\cC)}$. Let 
$$
t_0=\sup\set{t\in [0,1]\colon\mbox{there exists $\Gamma_t\in\mathcal{E}_S(\mathcal{C}[\phi(t)])$ that is $C^{k,\alpha}$ a.c.-isotopic to $\Gamma_0$}}.
$$
As $\Gamma_0$ is strictly stable, the projection $\Pi_{\Gamma_0}\colon\mathcal{ACE}^{k,\alpha}_n(\Gamma_0)\to \mathcal{V}_0$ which maps $[\mathbf{f}]$ to $\mathrm{tr}_\infty^1[\mathbf{f}]$ is a local diffeomorphism around $\mathbf{x}|_{\Gamma_0}$ and so $t_0>0$. 

Suppose $t_i\in [0,t_0)$ are such that $t_i\to t_0$ and that for each $i$ there exists $\Gamma_{t_i}\in\mathcal{E}_S(\cC[\phi(t_i)])$ that is $C^{k,\alpha}$ a.c.-isotopic to $\Gamma_0$. As $\cC[\phi(t_i)]\to\cC[\phi(t_0)]$ in $C^{k,\alpha}_{loc}(\Real^{n+1}\setminus\set{\mathbf{0}})$, Proposition \ref{CpctProp}, implies that, up to passing to a subsequence, $\Gamma_{t_i}\to\Gamma_{t_0}$ in $C^\infty_{loc}(\Real^{n+1})$ for an element $\Gamma_{t_0}\in\mathcal{E}_S(\cC[\phi(t_0)])$. Moreover, by \cite[Proposition 3.3]{BWProperness} and Lemma \ref{IsotopyLem}, $\Gamma_{t_0}$ is $C^{k,\alpha}$ a.c.-isotopic to $\Gamma_{t_i}$ and, thus, to $\Gamma_0$. In particular, it is enough to show $t_0=1$.  

If $t_0<1$, then $\mathbf{x}|_{\Gamma_{t_0}}$ cannot be a regular value of $\Pi_{\Gamma_0}$.  That is, $\Gamma_{t_0}$ is weakly stable. However, by Theorem \ref{StableIsotopyThm}, there is a sufficiently small $\epsilon>0$ so that for every $t$ with $|t-t_0|<\epsilon$ there is an element in $\mathcal{E}_S(\mathcal{C}[\phi(t)])$ that is $C^{k,\alpha}$ a.c.-isotopic to $\Gamma_{t_0}$.  This contradicts the definition of $t_0$ and so implies $t_0=1$ proving the proposition.
\end{proof}

\section{A mountain pass theorem for self-expanders} \label{MinmaxApp}
In the proof of Theorems \ref{MainThm} and \ref{AuxThm}, we used the following result which follows by combining \cite[Theorem 1.1]{BWMinMax}, \cite[Proposition 3.3]{BWProperness} and Lemma \ref{RegLem}.

\begin{thm}\label{MinMaxThm}
Fix an integer $k\geq 2$ and $\alpha\in (0,1)$. Let $\cC$ be a $C^{k+1}$-regular cone in $\mathbb{R}^{n+1}$ and assume either $2\leq n\leq 6$ or $\lambda[\cC]<\Lambda_n$. Suppose $\Sigma_-$ and $\Sigma_+$ are distinct strictly stable $C^{k,\alpha}_*$-asymptotically conical self-expanders with $\mathcal{C}(\Sigma_-)=\mathcal{C}(\Sigma_+)=\mathcal{C}$ and $\Sigma_-\preceq\Sigma_+$. Then there exists a $C^{k,\alpha}_*$-asymptotically conical self-expander $\Sigma_0\neq\Sigma_\pm$ with $\mathcal{C}(\Sigma_0)=\mathcal{C}$ and $\Sigma_-\preceq\Sigma_0\preceq\Sigma_+$.\end{thm}

\end{document}